\documentclass[10pt]{article}

\usepackage[dvipdf]{graphicx}
\usepackage[latin1]{inputenc}
\usepackage{latexsym}
\usepackage{amsmath,amssymb,amsfonts,amsthm}
\usepackage{xcolor}
\usepackage{mathrsfs}
\usepackage[colorlinks=true,citecolor=red,linkcolor=blue,urlcolor=blue,pdfstartview=FitH]{hyperref}

\newcommand{\ER}{\mathbb {R}}

\newcommand{\EN}{\mathbb {N}}

\newcommand{\PE}{\mathbb {P}}
\newcommand{\ES}{\mathbb{E}}
\newcommand{\abs[1]}{\lvert#1\rvert}

\newcommand{\bqn}{\begin{equation}}
\newcommand{\bqne}{\begin{equation*}}
\newcommand{\eqn}{\end{equation}}
\newcommand{\eqne}{\end{equation*}}

\newcommand{\Hzero}{\mathbf{(H_0)}}
\newcommand{\Hdeux}{\mathbf{(H_2)}}
\newcommand{\Hun}{\mathbf{(H_1)}}
\newcommand{\Hpun}{\mathbf{H'_1(\theta)}}
\newcommand{\Hpdeux}{\mathbf{H'_2(\theta)}}
\newcommand{\muchap}{\tilde{\mu}}
\newcommand{\kbis}{v}
\newcommand{\bee}{b}
\newtheorem{theorem}{\textnormal{\bf{T\scriptsize{HEOREM}}}}
\newtheorem{prop}{\textnormal{\bf{P\scriptsize{ROPOSITION}}}}

\newtheorem{lemme}{\textnormal{\bf{L\scriptsize{EMMA}}}}

\theoremstyle{definition}
\newtheorem{definition}{\textnormal{\bf{D}\scriptsize{EFINITION}}}
\theoremstyle{remark}

\newtheorem{Remarque}{\textnormal{\bf{R\scriptsize{EMARK}}}}

\def\restriction#1#2{\mathchoice
              {\setbox1\hbox{${\displaystyle #1}_{\scriptstyle #2}$}
              \restrictionaux{#1}{#2}}
              {\setbox1\hbox{${\textstyle #1}_{\scriptstyle #2}$}
              \restrictionaux{#1}{#2}}
              {\setbox1\hbox{${\scriptstyle #1}_{\scriptscriptstyle #2}$}
              \restrictionaux{#1}{#2}}
              {\setbox1\hbox{${\scriptscriptstyle #1}_{\scriptscriptstyle #2}$}
              \restrictionaux{#1}{#2}}}
\def\restrictionaux#1#2{{#1\,\smash{\vrule height .8\ht1 depth .85\dp1}}_{\,#2}} 

\author{Joaquin Fontbona\footnote{Department of Mathematical Engineering and Center for Mathematical Modeling,  UMI(2807) UCHILE-CNRS, Universidad de Chile,  { Casilla 170-3, Correo 3, Santiago-Chile.}  E-mail:\texttt{fontbona@dim-uchile.cl}.  }\, and 
Fabien Panloup\footnote{Institut de Math\'ematiques de Toulouse, Universit\'e Paul Sabatier \& INSA Toulouse, 135, av. de Rangueil, F-31077 Toulouse Cedex 4, France. E-mail: \texttt{fabien.panloup@math.univ-toulouse.fr}}
}

\title{Rate of convergence to equilibrium of fractional driven stochastic differential equations with some multiplicative noise}
\begin{document}
\maketitle
\begin{abstract}
{ We investigate the problem of the rate of convergence to equilibrium for ergodic stochastic differential equations driven by fractional Brownian motion with Hurst parameter $H>1/2$ and multiplicative noise component $\sigma$. When $\sigma$ is constant and for every $H\in(0,1)$, it was proved in \cite{hairer} that, under some mean-reverting assumptions, such a process converges to its equilibrium at  a rate of order $t^{-\alpha}$ where $\alpha\in(0,1)$ (depending on $H$). The aim of this paper is to extend such types of results to some multiplicative noise setting. More precisely, we show  that we can recover such convergence rates when   $H>1/2$ and   the inverse of the diffusion coefficient $\sigma$ is a Jacobian matrix.  The main novelty  of this work  is   a type of extension of Foster-Lyapunov  like techniques to this non-Markovian setting, which  allows us to put in place an asymptotic coupling scheme  such as in  \cite{hairer} without resorting to deterministic contracting properties.} 
\end{abstract}

\bigskip
\noindent \textit{Keywords}: Stochastic Differential Equations; Fractional Brownian Motion; Multiplicative noise; Ergodicity; Rate of convergence to equilibrium; Lyapunov function; Total variation distance.

\medskip
\noindent \textit{AMS classification (2010)}: 60G22, 37A25.

\section{Introduction}
Stochastic Differential Equations (SDEs) driven by a fractional Brownian motion (fBm) have been introduced to model 
random evolution phenomena whose noise has long range dependence properties. 
Indeed, beyond the historical motivations in Hydrology and Telecommunication for the use of fBm (highlighted e.g in \cite{MVn68}), recent applications of dynamical systems driven by this process include challenging issues in Finance \cite{Gua06}, Biotechnology~\cite{Odde-al96} or Biophysics~\cite{Jeon-al11,Kou08}.  

\noindent 
\smallskip The study of  the long-time behavior (under some stability properties) for fractional SDEs has been developed by Hairer \cite{hairer},  Hairer and Ohashi \cite{hairer2}, and by Hairer and Pillai \cite{hairer-pillai} (see also~\cite{Arnold98,crauel,GKN09} for another setting called 
random dynamical systems and \cite{cohen-panloup,cohen-panloup-tindel} for some results of approximations of stationary solutions)  who introduced a  suitable notion of stationary solutions for these a priori non-Markov SDE's and  extended some of the tools of the Markovian theory to this setting. In particular, criteria for uniqueness of the invariant distribution are provided in the three above papers in  different settings, namely: additive noise, multiplicative noise with $H>1/2$ and multiplicative noise with $H\in(1/3,1/2)$ (in an hypoelliptic context), respectively.\\

\noindent 
\smallskip When uniqueness holds for the invariant distribution, a challenging question is that of the rate of convergence to this equilibrium. In \cite{hairer}, the author proved that in the additive noise setting, the process converges in total variation to the stationary regime with a rate upper-bounded by  $ C_\varepsilon t^{-(\alpha-\varepsilon)}$ for any $\varepsilon>0$,  
with 
\bqn
\alpha=\begin{cases}\frac{1}{8}&\textnormal{if $H\in (\frac{1}{4},1)\backslash\left\{\frac{1}{2}\right\}$}\\
 H(1-2H)&\textnormal{if $H\in(0,\frac{1}{4}]$.}
\end{cases}
\eqn
But, to the best of our knowledge, no result of rate of convergence exists in the multiplicative setting. The aim of the current paper is to extend the results of \cite{hairer} to the multiplicative setting when $H>1/2$.

\smallskip
\noindent More precisely, we deal   with an  $\ER^d$-valued process $(X_t)_{t\ge0}$ which is a
solution to the following SDE
\begin{equation}\label{fractionalSDE0}
dX_t=b(X_t)dt+ \sigma(X_t)dB_t
\end{equation}
where $b:\ER^{d}\rightarrow\ER^d$ and $\sigma:\ER^{d}\rightarrow \mathbb{M}_{d,d} $  are (at least) continuous functions, and where $ \mathbb{M}_{d,d}$ is the  set of $d\times d$ real matrices. In~\eqref{fractionalSDE0}, $(B_t)_{t\ge0}$ is  a $d$-dimensional fractional Brownian motion  with Hurst parameter $H\in(\frac{1}{2},1)$,  $H$-fBm for short. Note that under some Hölder regularity assumptions on the coefficients (see $e.g.$ \cite{Nualart02,Coutin12} for background), (strong) existence and uniqueness hold for the solution to \eqref{fractionalSDE0} starting from $x_0\in\ER^d$. 

\noindent 
\smallskip Introducing the Mandelbrot-Van Ness representation of the fractional Brownian motion, 
\bqn\label{eq:mandel}
B_t=\alpha_H\int_{-\infty}^{0} (-r)^{H-\frac{1}{2}} \left(dW_{r+t}-dW_r\right),\quad t\ge0,
\eqn
where $(W_t)_{t\in\ER}$ is a two-sided $\ER^d$-valued Brownian Motion and $\alpha_H$ is a normalization coefficient depending on $H$, 
$(X_t, (B_{s+t})_{s\le 0})_{t\ge0}$ can be realized through a Feller transformation $({\cal Q}_t)_{t\ge0}$ on the product space $\ER^d\times{\cal W}_{\theta,\delta}$
 ($\theta\in(1/2,H)$ and $\theta+\delta\in(H,1)$) 
whose definition is recalled in \eqref{eq:HH} (we refer to \cite{hairer2} for more rigorous background on this topic). In particular,
an initial distribution of this dynamical system is a distribution $\mu_0$ on $\ER^d\times{\cal W}_{\theta,\delta}$. In probabilistic words, an initial distribution is the distribution of  a couple $(X_0, (B_s)_{s\le0})$ where $(B_s)_{s\le0}$ is an $\ER^d$-valued$H$-fBm on $(-\infty,0]$.

\smallskip
\noindent Then, such an initial distribution is classically called an
invariant distribution if it is invariant by the transformation ${\cal Q}_t$ for every $t\ge0$. However, the concept of uniqueness of invariant distribution is slightly different from the classical setting. Actually, we say that uniqueness of the invariant distribution holds if the stationary regime, that is,  the distribution ${\bar Q}\mu$ of the whole process $(X_t^\mu)_{t\ge0}$ with initial distribution $\mu$,  is unique;  in other words, this concept of uniqueness corresponds to the classical one up to identification   by  the equivalence relation: $\mu\sim\nu \Longleftrightarrow \bar{Q}\mu\sim\bar{Q}\nu$, see \cite{hairer2} for background.  In harmony with the previous concept, coupling two paths issued of $\mu_0$ and $\mu$, where the second one denotes an invariant distribution of $({\cal Q}_t)_{t\ge0}$, consists (classically) in finding a stopping time ${\tau_{\infty}}$
such that $(X_{t+{\tau_{\infty}}}^{\mu_0})_{t\ge0}= (X_{t+{\tau_{\infty}}}^{\mu})_{t\ge0}$. Thus, a rate of convergence in total variation can be deduced from bounds  established on 
$\PE({\tau_{\infty}}>t)$, $t\ge0$. 

\smallskip
\noindent Now, let us briefly recall the coupling strategy of \cite{hairer}. First, one classically waits that the paths get close. Then, at each trial, the coupling attempt is divided in two steps. First, one tries in Step 1 to stick or cluster the positions within an interval of length $1$. Then, in Step 2, one tries to ensure that the paths stay clustered until $+\infty$.  Actually, oppositely to the Markovian case where the paths stay naturally together after a clustering (by putting the same noise on each coordinate), the main difficulty here is that, due to the memory, staying together is costly. In other words, this property can be ensured only with help of a  non trivial coupling of the noises. We thus talk of \textit{asymptotic coupling}. If one of the two previous steps fails, we will begin a new attempt but only after a (long) waiting time which is called Step $3$. During this step, we again wait for the paths to get close, but also wait for the memory of the coupling cost to become sufficiently small, in order to start  a new trial only after  a weak influence  of the past is granted.

\smallskip
\noindent In the previous construction, the fact that $\sigma$ is constant is fundamental to ensure the two following properties:
\begin{itemize}
\item{} If two fBms $B^1$ and $B^2$ differ by a drift term, then two solutions $X^1$ and $X^2$ of \eqref{fractionalSDE0} respectively directed by 
$B^1$ and $B^2$ also differ by a drift term. This allows in particular to use Girsanov Theorem to build the coupling in Step 1.   
\item{} Under some ``convexity'' assumptions on the drift away from a compact {set}, two paths $X^1$ and $X^2$ directed by the same fBm (or more precisely, by two slightly different paths) get closer and the distance between the two paths can be controlled deterministically. 
\end{itemize}

In the present paper,  $\sigma$ is not constant and the two above properties are no longer valid. The challenge then is to  extend the applicability of the previous coupling scheme to such a situation. The replacement of  each of the above properties  requires us to deal with different (though related) difficulties.   In order to be able to extend the  Girsanov argument  used in Step 1 to a non constant $\sigma$, we will restrain ourselves to diffusion coefficients  for which some injective function  of two copies of the process differs by a drift term whenever their driving fBm do. A natural  assumption  on $\sigma$ granting the latter property is  that $x\mapsto\sigma^{-1}(x)$  is (well-defined and is)  a Jacobian matrix. This will be the setting of the present paper.

 As concerns a suitable substitution  of the second lacking property, a natural (but to our knowledge so far not explored) idea is  to attempt to extend Meyn-Tweedie techniques (see $e.g.$ \cite{DownMeynTweedie} for background) to the fractional setting. More precisely, even if the paths do not get closer to each other deterministically, one could expect that  some Lyapunov assumption could eventually  make the two paths return to some compact set simultaneously. {The main contribution of the present paper is to incorporate  such a  Lyapunov-type approach into the study of long-time convergence  in the fractional diffusion setting}. As one could expect, compared to the Markovian case, the problem is much more involved. Actually, the  return  time to a compact set after a (failed) coupling attempt does not only depend on   the positions of the processes after it, but also on all the past of the fBm. Therefore,  in order that  the coupling  attempt succeeds   with lower-bounded probability, one needs to establish some controls on the  past behavior of the fBms  that drive the two copies of the process,  conditionally to the failure of the previous attempts. This point is one of the main difficulties of the paper,  since, in the corresponding estimates, we  carefully have to take into account   all the deformations  of the distribution that  previously failed attempts induce. Then, we show that after a sufficiently long waiting time, conditionally on previous fails the probability that the two paths be in a compact set and  that the influence  of past noise on the future be controlled, is lower-bounded.  Bringing all the estimates together yields a global control of the coupling time and a rate of convergence which is similar to the one in \cite{hairer} in the additive noise case.

We notice that the application of the previous ideas to fractional SDE with more general diffusion coefficients can be considered. This  would in particular require to  extend a part of our computations and estimates to a framework where less regularity is available. Such an extension remains by the moment open.

\medskip

{In  Section \ref{section2} we detail our assumptions and state our main result, namely Theorem \ref{theo:principal}. The scheme of its proof, based on the previous described  coupling strategy,  is then given.  The proof of Theorem \ref{theo:principal} is achieved in Sections \ref{sec:step12}, \ref{sec:Kadmis}
and \ref{sec:prooftheoprinc}, which are  outlined at the end of Section \ref{section2}.}

\section{Assumptions and main result}\label{section2}
We begin by listing a series of notations and definitions.
\begin{itemize}
\item{} The scalar product and the Euclidean norm  on $\ER^d$ are respectively denoted by $(\,|\,)$ and $|\,.\,|$. 
\item{} The non explicit constants will be usually denoted by $C$ and may change from line to line.
\item{} The space ${\cal C}([0,+\infty),\ER^d)$ denotes the space of continuous functions on $[0,+\infty)$ endowed with the topology of uniform convergence on compact spaces. 
\item{} For some given $a,b\in\ER$, with $a,b$, $L^2([a,b],\ER^d)$ denotes the space of Lebesgue-measurable functions such that 
$\|g\|_{[a,b],2}=\sqrt{\int_a^b |g(s)|^2ds}<+\infty$.
\item{} For  some positive $\theta$ and $\delta$ such that $\theta\in(1/2,H)$ and $\theta+\delta\in(H,1)$, ${\cal W}_{\theta,\delta}$ denotes  the Polish space ${\cal W}_{\theta,\delta}$   which is the completion of ${\cal C}_0^{\infty}((-\infty,0],\ER^d)$ (the space of ${\cal C}^\infty$-functions $f:(-\infty,0]\rightarrow\ER^d$ with compact support and  $f(0)=0$)  for the norm
\bqn\label{eq:HH}
\|f\|_{{\cal W}_{\theta,\delta}}=\sup_{-\infty<s<t\le 0}\frac{|f(t)-f(s)|}{|t-s|^\theta(1+|t|^\delta+|s|^{\delta})}.
\eqn
\item{} For some real numbers $a<b$ and for $\theta\in(0,H)$, we denote by ${\cal C}^\theta([a,b],\ER^d)$ the set of functions $f:\ER_+\rightarrow\ER^d$ such that 
 $$ \|f\|_{\theta}^{a,b}=\sup_{a\le s< t\le b}\frac{ |f(t)-f(s)|}{(t-s)^\theta}<+\infty,$$

\item{} Let  $\sigma:\ER^{d}\rightarrow \mathbb{M}_{d,d} $  be a ${\cal C}^{1}$-function and $\gamma\in(0,1]$. We say that $\sigma $ is  $(1+\gamma)$-Lipschitz if  for every $i,j\in\{1,\ldots,d\}$, the following norm is finite:
\begin{equation}
  \label{eq:1+gammaLips}
  \|\sigma_{i,j}\|_{1 + \alpha}= \sup_{ x \in \mathbb R^d } |\nabla \sigma_{i,j}(x) | + \sup_{x, y \in \mathbb R^d}\frac{|\nabla \sigma_{i,j}(x)-\nabla \sigma_{i,j}(y)|}{|x-y|^\gamma},
\end{equation}
where for a given ${\cal C}^1$-function $f:\ER^d\rightarrow\ER$, $\nabla f=(\partial_{x_1} f,\ldots,\partial_{x_d} f)'$.
\item{} We also denote by ${\cal E}\!{\cal Q}(\ER^d)$ the set of {\em Essentially Quadratic} functions, that is ${\cal C}^1$-functions
$V:\ER^d\rightarrow (0,\infty)$ such that $\nabla V$ is Lipschitz continuous,
\[
 \liminf \frac{V(x)}{|x|^2}>0, \quad \mbox{  and }\quad 
  \abs[\nabla V]\le C \sqrt{V}.
 \]
 where $C$ is a positive constant. Note that these assumptions ensure that $\inf V=\min V$ is positive and that $\sqrt{V}$ is Lipschitz continuous (since it has a bounded gradient) which in turns implies that $V$ is subquadratic. 
 \end{itemize}
\smallskip
 
\noindent  Now, let us introduce the assumptions:

\smallskip
\noindent $\Hzero$: $b$ is a locally Lipschitz and sublinear function and  $\sigma$ is a bounded $(1+\gamma)$-Lipschitz continuous function $\gamma\in(\frac{1}{H}-1,1]$).

\smallskip
\noindent This condition ensures existence and uniqueness of solutions for \eqref{fractionalSDE0}. Note that the condition on the derivative of $\sigma$ only plays a role for uniqueness and in particular, is not fundamental for what follows. However, for the sake of simplicity, we choose to assume  this assumption throughout the paper.

Now, we turn to some more specific assumptions $\Hun$ and $\Hdeux$. The first one is a Lyapunov-stability assumption:

\smallskip
\noindent$\Hun$: There exists a function $V:\ER^d\rightarrow\ER$  of ${\cal E}\!{\cal Q}(\ER^d)$, there exist some positive $\beta_0$ and $\kappa_0$ such that
$$\forall\;x\;\in\ER^d,\quad (\nabla V(x)|b(x))\le \beta_0-\kappa_0 V(x).$$

{\begin{Remarque}
The above assumption will be used to ensure that the paths live with high probability in a compact set of $\ER^d$ (depending of the coercive function $V$). Note that in the classical diffusion setting, such a property holds with some less restrictive Lyapunov assumptions. Here, the assumptions essentially allow us to consider only (attractive) drift terms whose growth is linear at infinity. On the one hand, due to $\Hzero$, one can not consider drift terms with (strictly) superlinear growth at infinity and on the other hand, Assumption $\Hun$ combined with the fact that $V$ is subquadratic implies more or less that $b$ can not have (strictly) sublinear growth at infinity (this would be possible if $V$ had an exponential growth). These restrictions are mainly due to the lack of martingale property for the integrals driven by fBms.
\end{Remarque} }

\smallskip
\noindent
 Then, when the paths are in this compact set, one tries classically to couple them with positive probability. But, as mentioned before, the specificity of the non-Markovian setting is that the coupling attempts generate a cost for the future (in a sense made precise later). In order to control this cost (or, more precisely, in order that we can  couple the paths  using suitably controlled drift terms) we need the following assumption:

\smallskip
\noindent $\Hdeux$ $\forall x\in\ER^d$, $\sigma(x)$ is invertible and there exists a ${\cal C}^1$-function $h=(h_1,\ldots,h_d):\ER^d\rightarrow\ER^d$ such that
the Jacobian matrix $\nabla h=(\partial_{x_j} h_i)_{i,j\in\{1,\ldots,d\}}$ satisfies $\nabla h(x)=\sigma^{-1}(x)$ and such that $\nabla h$ is a locally Lipschitz function on $\ER^d$.

\smallskip
\noindent
\begin{Remarque}
$\rhd$ Under $\mathbf{(H_0)}$ and $\Hdeux$, $h$ is a global ${\cal C}^1$-diffeomorphism from $\ER^d$ to $\ER^d$. Indeed, under these assumptions,  $\nabla h$ is invertible everywhere and $x\mapsto [(\nabla h)(x)]^{-1}=\sigma(x)$ is bounded on $\ER^d$. The  property (which will be important in the sequel) then follows  from the Hadamard-Lévy theorem (see e.g. \cite{ruadulescu}). 

\smallskip
\noindent $\rhd$ As mentioned before, the main restriction here is to assume that $x\mapsto \sigma^{-1}(x)$ is a Jacobian matrix. However, there is no assumption on $h$ (excepted smoothness). In particular, $\sigma^{-1}$ does not need to bounded. This allows us to consider for instance some cases where $\sigma$ vanishes at infinity.

\smallskip
\noindent Let us exhibit some simples classes of SDEs  for which $\Hdeux$ is fulfilled. First, it contains the class of non-degenerated SDEs for which each coordinate is directed by one real-valued fBm. More precisely,  if for every $i\in\{1,\ldots,d\}$, 
$$dX_t^i=b_i(X_t^1,\ldots,X_t^d)dt+\sigma_i(X_t^1,\ldots,X_t^d) dB_t^i$$
where $\sigma_i:\ER^d\rightarrow\ER$ is a   ${\cal C}^1$ positive function, Assumption $\Hdeux$ holds. Next, let us also remark that,  since  $\nabla (P h)=P\nabla h$ for any square matrix $P$,  the following equivalence holds:
$$\textnormal{$\exists$ $h$ as in  $\Hdeux$  s.t. $\nabla h=\sigma^{-1}$ $\Longleftrightarrow$ $\exists$ $\tilde{h}$ as in  $\Hdeux$ and an invertible matrix $P$ s.t. $\sigma^{-1}=P\nabla \tilde{h}$},$$
Thus, assumption  $\Hdeux$ also holds true if:
\bqne
\sigma (x)=P {\rm Diag}\left(\sigma_1(x_1,\ldots,x_d),\ldots, \sigma_d(x_1,\ldots,x_d)\right)
\eqne
where $P$ is a given invertible $d\times d$-matrix and for every $i\in\{1,\ldots,d\}$ $\sigma_i$ is as before.
\end{Remarque}

We can now  state our main result. One denotes by ${\cal L}((X_{t}^{\mu_0})_{t\ge0})$ the distribution of the process on ${\cal C}([0,+\infty),\ER^d)$ starting from an initial distribution $\mu_0$
 and by $\bar{\cal Q}\mu$ the distribution of the stationary solution (starting from an invariant distribution $\mu$). The distribution $\bar{\mu}_0(dx)$ denotes the first marginal of $\mu_0(dx,dw)$.
\begin{theorem}\label{theo:principal} Let $H\in(1/2,1)$. Assume $\Hzero$, $\Hun$ and $\Hdeux$. Then, existence and uniqueness hold for the invariant distribution $\mu$ (up to equivalence). Furthermore, 
for every initial distribution $\mu_0$ for which there exists $r>0$ such that { $\int |x|^r \bar{\mu}_0(dx)<\infty$}, for each $\varepsilon>0$ there exists $C_{\varepsilon}>0$ such that 
$$\| {\cal L}((X_{t+s}^{\mu_0})_{s\ge0})-\bar{\cal Q}\mu\|_{TV}\le C_\varepsilon t^{-(\frac{1}{8}-\varepsilon)}.$$
\end{theorem}

\begin{Remarque}
In the previous result, the main contribution is the fact that one is able to recover the rates of the additive case. Existence and uniqueness results are not really new. However, compared with the assumptions of \cite{hairer2}, one observes that when  $x\mapsto\sigma^{-1}(x)$ is a Jacobian matrix (assumption which does not appear in \cite{hairer2}), our other assumptions are slightly less constraining. In particular, $b$ is assumed to be locally Lipschitz and sublinear (instead of Lipschitz continuous) and, as mentioned before, $x\mapsto \sigma^{-1}(x)$ does not need to bounded. Finally, remark that $\mathbf{(H_1)}$ is slightly different from the assumption on the drift of \cite{hairer} which   is a contraction condition 
out of a compact set: for any $x,y$, $(b(x)-b(y)|x-y)\le \beta_0-\kappa_0|x-y|^2$. 
This means that even in the constant setting, our work can cover some new cases. 
For instance, if $d=2$ and $b(z)=-z-\rho \cos(\theta_z) z^\perp$ 
(where $\rho\in\ER$, $\theta_z$ is the angle of $z$ and $z^\perp$ is its normal vector), 
Assumption $\mathbf{(H_1)}$ holds with $V(z)=1+|z|^2$ whereas one can check that the contraction 
condition is not satisfied if $\rho>2$. 

\end{Remarque}

\subsection{Scheme of coupling}
As explained before, the proof of Theorem \ref{theo:principal} is based on a coupling strategy similar to that of \cite{hairer}. 
Let $(B^{1}_t)_{t\in\ER}$ and $(B^{2}_t)_{t\in\ER}$ denote two fractional Brownian motions with Hurst parameter $H>1/2$. Then, denote by $(X_t^{1},X_t^{2})$, a couple  of solutions to \eqref{fractionalSDE0}:

\bqn\label{eq:eqcoup}
\begin{cases}
dX_t^{1}= b(X_t^{1})dt+\sigma(X_t^{1}) dB^{1}_t\\
dX_t^{2}= b(X_t^{2})dt+\sigma(X_t^{2}) dB^{2}_t
\end{cases}
\eqn
with initial conditions $(X_0^1,(B_t^1)_{t\le0})$ $(X_0^1,(B_t^2)_{t\le0})$. We denote by  $({\cal F}_t)_{t\ge0}$ the usual augmentation of the filtration  $(\sigma(B_s^1,B_s^2, (X_0^1,X_0^2))_{s\le t})_{t\ge0}$.   
To begin the coupling procedure without ``weight of the past'', we will certainly assume that $$(B^1_t)_{t\le0}=(B^2_t)_{t\le0}$$
and that the initial distribution $\tilde{\mu}$  of $(X^1,X^2)$ is of the form
\bqn\label{eq:mutilde}
\tilde{\mu}(dx,dw)=\mu_1(w,dx_1)\mu_2(w,dx_2)\PE_H(dw)
\eqn
where $\PE_H$ denotes the distribution of a fBm $(B_t)_{t\le0}$ on ${\cal W}_{\gamma,\delta}$ and the transitions probabilities $\mu_1(.,dx_1)$ and $\mu_2(.,dx_2)$
correspond respectively to the conditional distributions of $X_0^1$ and $X_0^2$ given $(B_t^1)_{t\le0}$. The processes  $(B^{1}_t)_{t\in\ER}$ and $(B^{2}_t)_{t\in\ER}$ can be realized through the Mandelbrot-Van Ness representation (see \eqref{eq:mandel})  with the help of some  two-sided Brownian motions respectively denoted by $W^1$ and $W^2$.  In particular, the filtration $({\cal F}_t)_{t\ge0}$  is also generated by  $(\sigma(W_s^1,W_s^2, (X_0^1,X_0^2))_{s\le t})_{t\ge0}$.

Furthermore, we will assume in all the proof that on $[0,\infty)$,  $W^1$ and $W^2$ (resp. $B^1$ and $B^2$)
differ by  a (random) drift term denoted by $g_w$ (resp. $g_B$): 
\bqn\label{eq:gw}
dW^2_t=dW^1_t+g_w(t) dt\quad \textnormal{and}\quad dB_t^2=dB_t^1+g_B(t) dt.
\eqn
Note that the functions $g_w$ and $g_B$ are linked by some inversion formulas (see \cite{hairer}, Lemma 4.2 for details). 

\smallskip
The idea is to build $g_w$ (resp. $g_B$) in order to stick $X^{1}$ and $X^{2}$. We set 
$$\tau_\infty:=\inf\{t\ge0,\; X_s^{1}=X_s^{2}\;\forall s\ge t\}.$$ 
As usual, this coupling  will be achieved after a series of trials. As mentioned in the introduction, each trial is decomposed in three steps:
\begin{itemize}
\item{} Step 1: Try to couple the positions with a \textit{controlled cost} (in a sense made precise below).
\item{} Step 2 (specific to non-Markov processes): Try to keep the paths  fastened together.
\item{} Step 3: If Step 2 fails, wait a sufficiently long time in order that in the next trial,  Step 1 be achieved with a \textit{controlled cost} and with  (uniformly lower-bounded away from $0$) probability. During this step, we suppose that $g_w(t)=0$.
\end{itemize}
Let us make  a few precisions:

\smallskip
\noindent $\rhd$ We denote by $\tau_0\ge0$ the beginning of the first trial and by $\tau_k$, $k\ge1$, the end of each trial. If $\tau_k=+\infty$,
the coupling tentative has been successful. Otherwise, $\tau_k$ is the end of Step 3 of trial $k$.
We will assume that 
$$\forall t\in(-\infty,\tau_0],\quad W^1_t=W^2_t \quad a.s. \quad \textnormal{or equivalently that}\quad g_w(t)=g_B(t)=0 \quad\textnormal{on $[-\infty,\tau_0]$.}$$

\smallskip
\noindent $\rhd$ About Step 1 and the ``controlled cost'': Step 1  is  carried out on each interval $[\tau_{k-1},\tau_{k-1}+1]$. The ``cost'' of coupling is represented by the function $g_w$ that one needs to 
build on $[\tau_{k-1},\tau_{k-1}+1]$ in order to get  $X^{x_1}$ and $X^{x_2}$  stuck together at time $\tau_{k-1}+1$. Oppositely to the Markovian case, this cost does not only depend  on the positions of  $X_{\tau_{k-1}}^{1}$ and $X_{\tau_{k-1}}^{2}$  but also on the past of the Brownian motions,  which have a (strong) influence on the dynamics of $B^{1}$ and $B^{2}$. This is the reason why one needs in Step 3 to wait enough before beginning a new attempt of coupling.
\medskip

In \cite{hairer}, the ``controlled cost'' concept is called ``admissibility'' (see Definition 5.6). Here, we slightly modify it and we will say that one is in position to attempt a coupling if the system is \textit{$(K,\alpha)$-admissible}. We define this concept below but need before to introduce notations. For $T\ge0$ and  a measurable function $g:\ER\rightarrow\ER$, we denote by ${\cal R}_T g$ the function defined {(when it makes sense)} by 
$$({\cal R}_T g)(t)=\int_{-\infty}^0 \frac{t^{\frac{1}{2}-H} (T-s)^{H-\frac{1}{2}}}{t+T-s} g(s) ds,\quad t\in(0,+\infty).$$
Let $g_w$ be the (random) function defined by \eqref{eq:gw}.  For a positive time $\tau$, we denote by $g_w^\tau$ the function defined by $g_w^\tau(t)=g_w(t+\tau)$, $t\in\ER$.

\medskip The following definition is relative to a fixed  $\theta\in(1/2,H)$.

\begin{definition} Let $K$ and $\alpha$ be some positive constants and  $\tau$ denote stopping time with respect to  $({\cal F}_t)_{t\in\ER}$.   We say that the  system  is $(K,\alpha)$-admissible at time $\tau$  if $\tau(\omega)<+\infty$ and if $(X_\tau^{1}(\omega),X_\tau^{2}(\omega), (W^1(\omega),W^2(\omega))_{t\le \tau})$ satisfies:
\bqn\label{hairerassumpcond}
 \sup_{T\ge0} \int_0^{+\infty} (1+t)^{2\alpha}|({\cal R}_T g_w^\tau)(t)|^2dt\le 1.
\eqn
and 
\bqn\label{Kadmiscond}
 |X_\tau^{1}(\omega)|\le K,\quad |X_\tau^{2}(\omega)|\le K, \quad  \varphi_{\tau,\varepsilon_\theta}(W^1(\omega))\le K \quad\textnormal{and}\quad \varphi_{\tau,\varepsilon_\theta}(W^2(\omega))\le K,
\eqn
where   $\varepsilon_\theta=\frac{H-\theta}{2}$ and for a given positive $\varepsilon$,
$$\varphi_{\tau,\varepsilon}(w)=\sup_{\tau\le s\le t\le \tau+1}|\frac{1}{t-s}\int_{-\infty}^{\tau-1} (t-r)^{H-\frac{1}{2}}-(s-r)^{H-\frac{1}{2}} dw_r| +\|w\|_{\frac{1}{2}-\varepsilon}^{\tau-1,\tau}$$

\end{definition}
If these two conditions hold, we will show that the coupling attempt is successful with lower-bounded probability. Thus,  we will need to ensure that at each time $\tau_k$,  the $(K,\alpha)$-admissibility also holds with lower-bounded probability. We set $\Omega_{K,\alpha,\tau}=\Omega_{\alpha,\tau}^1\cap\Omega_{K,\tau}^2$
where
\bqn\label{hairerassump}
\Omega_{\alpha,\tau}^1:=\left\{\omega, \tau(\omega)<+\infty,\;\sup_{T\ge0} \int_0^{+\infty} (1+t)^{2\alpha}|({\cal R}_T g_w^\tau)(t)|^2dt\le 1\right\},
\eqn
and 
\bqn\label{Kadmis}
\Omega_{K,\tau}^2:=\left\{\omega,\tau(\omega)<+\infty,\; |X_\tau^{1}|\le K, |X_\tau^{2}|\le K, \varphi_{\tau,\varepsilon_\theta}(W^1)\le K, \varphi_{\tau,\varepsilon_\theta}(W^2)\le K\right\},
%
\eqn
The novelty here is the event defined in \eqref{Kadmis}. Since, contrarily to the additive noise case, we  are not able to reduce here the distance between the positions deterministically, we ask  $X^{x_1}$ and $X^{x_2}$ to be in the compact set $\bar B(0,K)=\{y , |  y |\leq K\}$ with positive probability. The same type of assumption is needed on the past of the fractional Brownian motion (which is represented by the functionals $\varphi_{\tau,\varepsilon_\theta}(W^j)$, $j=1,2$). Note that,  oppositely to the event $\Omega_{\alpha,\tau}^1$, which comes from \cite{hairer}, $\Omega_{K,\tau}^2$ can certainly not have a probability equal to $1$.  
\smallskip 
\noindent We will attempt the coupling  on $[\tau_{k-1},\tau_{k-1}+1]$ only if $\omega\in\Omega_{K,\alpha,\tau_{k-1}}$. Otherwise, we set $g_w(t)=0$ on $[\tau_{k-1},\tau_{k-1}+1]$ (and,  in this case, we certainly say that Step 1 fails).

\smallskip
\noindent $\rhd$ If Step 1 fails (which includes the case where one does not attempt the coupling), one begins Step 3 (see below). Otherwise,
one begins Step 2. Step 2 is in fact a series of trials on some intervals $I_\ell$ with length 
\bqn\label{eq:defc2}
|I_\ell|=c_2 2^\ell
\eqn
where $c_2$ is a constant larger than one which will be calibrated later on. More precisely, one successively tries to keep $X^{1}$ and $X^{2}$ as being equal on intervals $[\tau_{k-1}+1+c_2\sum_{u=1}^{\ell-1} 2^{k},\tau_{k-1}+1+c_2\sum_{u=1}^{\ell} 2^k]$ (with the convention $\sum_{\emptyset}=0$).
The exponential increase of the length of the intervals will be of first importance to ensure the success of Step 2.

\smallskip
\noindent $\rhd$   If Step 2 fails at trial $\ell$ with $\ell\ge0$ ($\ell=0$ corresponds to the case where  Step 1 fails), one begins Step 3.  We denote by $\tau_{k}^{3}$, the beginning of Step 3. As mentioned before, the aim on this interval is to wait a sufficiently long time  in order to be in an $(K,\alpha)$-admissible state with positive probability (Step $3$ ends at time  $\tau_k$, $i.e.$ at the beginning of the next attempt). This has two natural consequences. On the one hand, one assumes that
\bqn\label{eq:assumpstep3}
g_w(t)=0 \quad \textnormal{on $[\tau_{k}^3,\tau_k]$,}\quad \mbox{  so that }\quad W^1_t-W^1_{\tau_{k}^3}=W^2_t-W^2_{\tau_{k}^3} \quad \textnormal{on $[\tau_{k}^3,\tau_k]$}.
\eqn
On the other hand, the waiting time will strongly depend on the length of Step 2. Longer is Step 2, longer is the waiting time. We set
\bqn\label{eq:akl}
{\cal A}_{k,\ell}=\{\textnormal{At trial $k$, Step 2 fails after $\ell$ attempts}\}.
\eqn
We assume in the sequel that 
\bqn\label{durationstep3}
\forall \omega\in{\cal A}_{k,\ell}, \quad \tau_k-\tau_{k}^3=\Delta_3(\ell,k)\quad\textnormal{with}\quad \Delta_3(\ell,k):=c_3 a_k 2^{\beta\ell}
\eqn
where $c_3\ge 2 c_2$, $\beta\in[1,+\infty)$ and $(a_k)_{k\ge1}$ is an increasing  deterministic sequence.  We will calibrate these quantities later (see Proposition \ref{lemme:step3}). {At this stage, we can however remark a useful property for the sequel: conditionally to ${\cal A}_{k,\ell}$, the length of each step is deterministic.}
We are now ready to begin the proof. In Section \ref{sec:step12}, we focus on Steps 1 and 2 and prove that we can achieve the coupling scheme in such a way that  for every positive $K$ and $\alpha$, the probability of coupling can be lower-bounded by a constant which does not depend on $k$.
Then, in Section \ref{sec:Kadmis}, we focus on the $(K,\alpha)$-admissibility condition. In particular,  we show that for $K$ large enough, \eqref{Kadmis} holds with high probability (which does not depend on $k$). Finally, in Section \ref{sec:prooftheoprinc}, we prove Theorem \ref{theo:principal}.
 
%
%
%

\section{Lower-bound for the successful-coupling probability}\label{sec:step12}
In this section, we detail the construction of Steps 1 and 2 with the aim of proving that if the system is $(K,\alpha)$-admissible at time $\tau_{k-1}$, then
the probability that $\Delta \tau_{k}:=\tau_k-\tau_{k-1}$ be infinite ($i.e.$ that the coupling be successful) can be lower-bounded.
The main result of this section is the next proposition.
\begin{prop} \label{prop:imp1} Let $H>1/2$. Assume that $\mathbf{(H_0)}$ and  $\Hdeux$ hold. Then, for every $K>0$ and $\alpha\in(0,H)$, Steps $1$ and $2$ can be achieved in such a way that there exists ${\delta}_0$ and $\delta_1$ in $(0,1)$ such that for every $k\ge1$, $\delta_0\le \PE(\Delta \tau_{k}=+\infty| \Omega_{K,\alpha,\tau_{k-1}})\le 1-{\delta}_1$. Furthermore, $\delta_1$ can be chosen independently of $K$.
\end{prop}
The (uniform) upper-bound is almost obvious. Actually, at  the beginning of Step $1$, it is always possible (if necessary) to attempt the coupling with probability $1-\delta_1$ only (and to put $W^1=W^2$ otherwise). This upper-bound may appear of weak interest but in fact, it will play an important role in Section 4.

\smallskip
\noindent The lower-bound is a consequence of the combination of Equation \eqref{eq:productinfty} with Lemmas \ref{lemma:step1} and \ref{lemme:step2.2} below.
\subsection{Step 1}
\begin{lemme}\label{lemma:step1} Assume that $\mathbf{(H_0)}$ and $\Hdeux$ hold. Let $K$ and $\alpha$ denote two positive constants and  $\theta\in(1/2,H)$ be fixed. Then,  for each $k\geq 1$, $(W^1,W^2)$ can be built on $[\tau_{k-1},\tau_{k-1}+1]$ in such a way that  the following properties hold:
\begin{itemize}
\item[(a)] There exists $\tilde{\delta}_0>0$ depending only on $K$,  $\alpha$ and  $\theta\in(1/2,H)$ such that for all  $k\ge0$,  
$$\PE(X_{\tau_{k-1}+1}^1=X_{\tau_{k-1}+1}^2| \Omega_{K,\alpha,\tau_{k-1}})\ge \tilde{\delta}_0.$$
\item[(b)] There exists $C_K>0$  such that $\int_{\tau_{k-1}}^{\tau_{k-1}+1}|g_w(s)|^2ds\le C_K\quad a.s.$
\item[(c)] If Step $1$ is successful, $t\mapsto g_B(t)$ is a ${\cal C}^1$-function on $[\tau_{k-1},\tau_{k-1}+1]$  such that
$$\sup_{t\in[0,1]} |g_B(t)|\le C\quad\textnormal{and}\quad \forall t\in[1/2,1],\quad g_B(t+\tau_{k-1})=0.$$
where $C$ is a deterministic constant which does not depend on $k$.
\end{itemize}
\end{lemme}
\begin{proof} 

$(a)$ The proof of this statement is divided in five parts:

\smallskip
\noindent 
\textbf{$(i)$} Let $\theta\in(1/2,H)$ and set $\varepsilon_\theta=\frac{H-\theta}{2}$. Let $\tilde{K}$ be a positive constant. Then, there exists a deterministic constant  depending only on $\theta$, $K$ and $\tilde{K}$  denoted $C(K,\tilde{K})$ such that
$$  \forall \omega\in\Omega_{K,\tau_{k-1}}^2\cap\{\|W^1\|_{\frac{1}{2}-\varepsilon_\theta}^{\tau_{k-1},\tau_{k-1}+1}\le \tilde{K}\},\quad \sup_{t\in[0,1]} |X_{\tau_{k-1}+t}^1(\omega)|\le C(K,\tilde{K}).$$ 
The proof of this property (whose arguments are close to some in the next sections) is given in  Appendix A. Notice that $C(K,\tilde{K})\geq K$ since $|X_{\tau_{k-1}+t}^1|\leq K$
on $\Omega_{K,\tau_{k-1}}^2$. 

\smallskip
\noindent \textbf{$(ii)$} Building a function $g_B$ to couple  $(X_t^1)$ and $(X_t^2)$ at time $\tau_{k-1}+1$:  First, note that this step is strongly based on Assumption $\Hdeux$ and that  the construction is a modified version of Lemma 5.8 of \cite{hairer}.  For a given past on $(-\infty,\tau_{k-1}]$ and a given innovation path 
$(W^1_{t+\tau_{k-1}},W^2_{t+\tau_{k-1}})_{t\in[0,1]}=(w_1(t)+W^1_{\tau_{k-1}},w_2(t)+W^2_{\tau_{k-1}})_{t\in[0,1]}$ of ${\cal C}^{\frac{1}{2}-{\varepsilon_\theta}}([0,1],\ER^d)^2$ (note that  $w_1(0)=w_2(0)=0$), set $(x^{w_1}(t),x^{w_2}(t))=(X^{1}_{t+\tau_{k-1}},X^{2}_{t+\tau_{k-1}})$, $t\in[0,1]$.   Let then $(B^{w_1}_t,B^{w_2}_t)_{t\in[0,1]}:=(B^1_{t+\tau_{k-1}},B^{2}_{t+\tau_{k-1}})_{t\in[0,1]}$ denote the corresponding fBMs, defined as in \eqref{eq:mandel}.

The aim now is to build, conditionally to  $\Omega_{K,\alpha,\tau_{k-1}}$ and  $(X_{\tau_{k-1}}^{1},X_{\tau_{k-1}}^{2}, (W^1,W^2)_{t\le {\tau_{k-1}}})$,  a function on $[0,1]$ denoted by $g_h$, such that whenever $w_2(t)=w_1(t)+\int_0^t g_h(s)ds$ for every $t\in[0,1]$,   then $x^{w_1}(1)=x^{w_2}(1)$.
In fact, it is more convenient to build an  associated function $f_h$ such that  $dB^{w_2}_t=dB^{w_1}_t+ f_h(t) dt$ (see \eqref{eq:gw} for background).   

\smallskip
\noindent With the previous notations, we see from Assumption $\Hdeux$ and  a change of variable formula for Hölder functions with exponent larger than $1/2$ (see $e.g.$  \cite{zahle}, Theorem 4.3.1) that such a function $f_h$ should satisfy, for every $t\in[0,1]$,  
\bqn\label{eq:zahledev}
 h(x^{w_2}(t))-h(x^{w_1}(t))= h(x_2)-h(x_1)+\int_0^t \nabla h b(x^{w_2}(u))-\nabla h b(x^{w_1}(u))du+\int_0^t f_h(u) du
 \eqn
 where $x_i=x^{w_i}(0)$, $i=1,2$. The idea is then to build  $t\mapsto f_h(t)$ as an adapted process  such that  the distance between $h(x^{w_2}(t))$ and $h(x^{w_1}(t))$  decreases to $0$ in the interval $[0,1]$. Due to the fact that $\nabla h$ is only supposed to be locally Lipschitz continuous, such a construction will indeed  be possible with a controlled cost only if $(x^{w_1}(t))_{t\in[0,1]}$ lies in a compact set of $\ER^d$.
 
  \noindent  For a given  $a\geq K$, we thus introduce a ``localizing''  ${\cal C}^1$-diffeomorphism $\varpi_a:\ER^d\rightarrow\varpi_a(\ER^d)$ with the following properties:
$$ \restriction{\varpi}{\bar{B}(0,a)}=\restriction{Id}{\bar{B}(0,a)},\quad \|\varpi\|_\infty\le a+1\quad \textnormal{and $\exists$ a norm $\|.\|$ on $\mathbb{M}_{d,d}$ such that }\quad \|\nabla \varpi\|\le 1.$$
We set $h_a:=h\circ \varpi_a$ and introduce a system $(y_a^1(t),y_a^2(t))_{t\ge0}$  \textit{companion}   to $(h(x^{w_1}(t), h(x^{w_2}(t))_{t\ge0}$ defined by 
$(y_a^{1}(0),y_a^{2}(0))=(h_a(x_1),h_a(x_2))=(h_a(x^{w_1}(0)),h_a(x^{w_2}(0)))$ and
\begin{equation*}
\begin{cases} dy_a^1(t)=(\nabla h_a b_a)(h^{-1}(y_a^1(t))dt+ dB_t^{w_1}\\
dy_a^2(t)=(\nabla h_a b_a)(h^{-1}(y_a^2(t)))dt + (dB_t^{w_1}+f_{h}(t)dt)
\end{cases}
\end{equation*}
where $b_a:\ER^d\rightarrow\ER^d$ is a localization of $b$, $i.e.$ a Lipschitz bounded continuous function such that $b_a(x)=b(x)$ on ${\bar{B}(0,a)}$ and 
$f_{h}$ is the function defined as follow.  For a given $a$, we set
\bqn\label{def:fh}
f_{h}(t)=-\kappa_1^a \rho_{h_a}(t)-\kappa_2\frac{\rho_{h_a}(t)}{\sqrt{|\rho_{h_a}(t)|}}\quad\textnormal{(with the convention $\frac{\rho}{\sqrt{|\rho|}}=0$ if $\rho=0$)}
\eqn
where $a$ and $\kappa_2$ are positive numbers to be fixed, $\rho_{h_a}$ is the unique solution starting from $\rho_{h_a}(0)=h_a(x_2)-h_a(x_1)$ 
to the equation:
\bqn\label{eq:rhoh}
\frac{d\rho_{h_a}}{dt}= F_a(t,\rho_{h_a}(t))
\eqn
with
\begin{equation*} F_a(t,\rho)=(\nabla h_a b_a) (h^{-1}( y_a^1(t)+\rho(t)))-(\nabla h_a b_a) (h^{-1}(y_a^1(t))) -\kappa_1^a \rho_{h_a}(t)-\kappa_2\frac{\rho_{h_a}(t)}{\sqrt{|\rho_{h_a}(t)|}},
\end{equation*}
and  
\bqn\label{eq:kappa1}
\kappa_1^a=\sup_{y_1,y_2\in \ER^d}\frac{|(\nabla h_a b_a)( h^{-1}(y_2))-(\nabla h_a b_a)(h^{-1}(y_1))|}{|y_2-y_1|}.
\eqn
Observe that  $\kappa_1^a$ is finite for each $a>0$ since $h^{-1}$ is Lipschitz continuous on $\ER^d$ and $\nabla h_a$ and $b_a$ are bounded Lipschitz continuous.  Moreover, $\rho_{h_a}(t)$ is uniquely defined  (at least) on some maximal interval $[0,t_0)$ such that $0<t_0\leq \infty$ and  $|\rho_{h_a}(t)|>0$ on $[0,t_0)$. Since by  definition of  $\kappa_1^a$ one has
$$\forall t\in [0,t_0) ,\quad \frac{d|\rho_{h_a}(t)|^2}{dt}\le -2\kappa_2 |\rho_{h_a}(t)|^{\frac{3}{2}},$$
the function $t\mapsto |\rho_{h_a}(t)|^2$ is non-increasing and, by a  standard computation,  we see that $t_0\le 2\sqrt{\rho_{h_a}(0)}/\kappa_2$. Moreover,  $\rho_{h_a}$  can then  be globally defined in a unique way   such  that
\bqn \label{eq:contrhoy}
|\rho_{h_a}(t)|\le \begin{cases} (\sqrt{|\rho_{h_a}(0)|}-\frac{\kappa_2}{2} t)^2& \textnormal{if $t\le 2\sqrt{\rho_{h_a}(0)}/\kappa_2$}\\
0& \textnormal{if $t\ge 2\sqrt{\rho_{h_a}(0)}/\kappa_2$}.
\end{cases}
\eqn
Hence,  $f_{h}$ (and thus $y_a^2$) is well-defined on $[0,\infty)$.  
 It follows from this construction that 
$y_a^{2}(t)-y_a^{1}(t)=\rho_{h_a}(t)$.
Now set
$$t_a:=1\wedge \inf\{t, \max(|x^{w_1}(t)|, |x^{w_2}(t)|)>{a}\}.$$
We observe   from \eqref{eq:zahledev} that,  since $h_a(x)=h(x)$ and $b_a(x)=b(x)$ on $\bar{B}(0,a)$, one has
\bqn\label{condrhohh}
\forall t\leq t_a,\quad y_a^1(t)=h(x^{w_1}(t)),\quad y_a^2(t)=h(x^{w_2}(t))\quad\textnormal{and}\quad \rho_{h_a}(t)=h(x^{w_2}(t))-h(x^{w_1}(t))
\eqn
(notice also that $t_a>0$ on $ \Omega_{K,\alpha,\tau_{k-1}}$ since $a\geq K$). Set now
$\bar{K}=\sup_{x_1,x_2\in \bar{B}(0,K)}|h(x_2)-h(x_1)|$  and, for every $M\geq K$, 
$$A_M:=\{w_1\in{\cal C}^{\frac{1}{2}-{\varepsilon_\theta}}([0,1],\ER^d): w_1(0)=0\textnormal{ and} \sup_{t\in[0,1]}|x^{w_1}(t)|\leq M \,,  \forall \omega\in\Omega_{K,\tau_{k-1}}^2\}.$$

We are going to show that,  for suitably chosen  $a=a(M,K)\geq K$ and $\kappa_2: =4\sqrt{\bar{K}}$,  it holds   on  $\Omega_{K,\alpha,\tau_{k-1}}$  that
\begin{equation}\label{w1AMx=x}
\forall w_1\in A_M,  \forall t\in[1/2,1],\quad x^{w_1}(t)=x^{w_2}(t).
\end{equation}



\noindent From \eqref{condrhohh}, the definition of  $t_a$,  the decrease of $|\rho_{h_a}|$ and the fact that  $|\rho_{h_a}(0)|\le \bar{K}$   on  $   \Omega_{K,\alpha,\tau_{k-1}} $, we get on that event  that   
$$\forall w_1\in A_M , \,   \sup_{t\in [0,1]} |h(x^{w_2}(t))|\le \bar{h}_M+\bar{K} \, ,$$
with  $\bar{h}_M:=\sup_{x\in \bar{B}(0,M)} |h(x)|$. 
Since $h$ is an  homeomorphism, $h^{-1}(\bar{B}(0,\bar{h}_M+\bar{K}))$ is a compact set and so 
$\hat{C}(M,K):=\sup\{|x|:  |h(x)|\le  \bar{h}_M+\bar{K}\}<+\infty$. 
We deduce that, for $a=a(M,K):=\max\{M,\hat{C}(M,K)\}+1$ it holds on    $ \Omega_{K,\alpha,\tau_{k-1}}$ that:
$$\forall w_1\in A_M, \forall t\in[0,1],\quad \max(|x^{w_1}(t)|,|x^{w_2}(t)|)\le a ,\mbox{ hence } t_a=1. $$
Thus,  for all $ w_1\in A_M$, identities \eqref{condrhohh} hold for all $t\in [0,1]$   on the event  $   \Omega_{K,\alpha,\tau_{k-1}} $.  The choice  $\kappa_2=4\sqrt{\bar{K}}$ then ensures  that $2\sqrt{\rho_h(0)}/\kappa_2\le 1/2$ which, by the previous,  yields \eqref{w1AMx=x}.

\smallskip 

\noindent This also certainly   implies that,  on the event  $   \Omega_{K,\alpha,\tau_{k-1}} $ on has  $f_{h}(t)=0$  on $[1/2,1]$  for all $w_1\in A_M$ (this fact  will be used in Step 2). From now on, we will  assume that $f_h$ is defined by \eqref{def:fh} with $a=a(M,K)$ and $\kappa_2=4\sqrt{\bar{K}}$.\smallskip

\noindent $(iii)$ About $f_h$ and $g_h$:  let $\omega\in\Omega_{K,\alpha,\tau_{k-1}}$ and $w_1 \in{\cal C}^{\frac{1}{2}-\varepsilon_\theta}([0,1],\ER^d)$ and consider the ${\cal C}^1$-function $(f_h(t))_{t\in[0,1]}$ built above. Given this function, let us recall how one defines a function $(g_{h}(t))_{t\in[0,1]}$ which is such that 
$$ dw_t^2=dw_t^1+g_{h}(t) dt \;\textnormal{on $[0,1]$} \Longrightarrow dB_t^{w_2}=dB_t^{w_1}+ f_h(t)dt\quad \;\textnormal{on $[0,1]$}.$$
The function $g_B$ of \eqref{eq:gw} being known on $(-\infty,\tau_{k-1}]$, one can define it on $(-\infty,\tau_{k-1}+1]$ by setting $g_B(t)=f_h(t-\tau_{k-1})$ on $[\tau_{k-1},\tau_{{k-1}}+1]$. By an inversion formula (see $(4.11a)$ of \cite{hairer}), one obtains a unique $g_w$  on $(-\infty,\tau_{k-1}+1]$ (where $g_w$ is defined in \eqref{eq:gw}). Then, $g_h$ can be defined by $g_h(t)=g_w(t+\tau_{k-1})=g_w^{\tau_{k-1}}(t)$, $t\in[0,1]$. In fact, it can be shown (see proof of Lemma 5.9  of \cite{hairer} for details) that the function  $g_h$ is given by 
\bqn\label{eq:explicitgh}
g_h(t)=C{\cal R}_0 g_w^{\tau_{k-1}}(t)+\alpha_H\frac{d}{dt}\left(\int_{0}^{t} (t-s)^{{\frac{1}{2}-H}} f_h(s)ds\right),\quad t\in(0,1].
\eqn

\smallskip
\noindent Note that, thanks to the $(K,\alpha)$-admissibility condition and to the differentiability of $f_h$, the function $g_h$ is measurable and integrable on $[0,1]$. We can thus define $\varphi: {\cal C}^{\frac{1}{2}-\varepsilon_\theta}([0,1],\ER^d)\rightarrow{\cal C}^{\frac{1}{2}-\varepsilon_\theta}([0,1],\ER^d)$ by 
\bqn\label{defvarphi1}
\varphi(w_1)_t=w_1(t)+\int_0^t g_{h}(s)ds,\quad t\in[0,1].
\eqn
Using again that  $w_1\mapsto   y_a^1$ is continuous from ${\cal C}^{\frac{1}{2}-\varepsilon_\theta}([0,1],\ER^d)$ to ${\cal C}^{\theta}([0,1],\ER^d)$, one can check that the mapping  $\varphi$ is measurable on ${\cal C}^{\frac{1}{2}-\varepsilon_\theta}([0,1],\ER^d)$. Furthermore, it is bijective with measurable inverse $\psi$ defined as follows: for a given ${w}_2$, denote by $(\tilde{y}_a^2(t))_{t\ge0}$ the solution to $d\tilde{y}_a^2(t)=(\nabla h_a b_a)(h^{-1}(\tilde{y}_a^2(t)))+ dB_t^{w_2}$ starting 
from $x_2$. One thus defines $\tilde{\rho}_{h_a}$ as the solution to $\frac{d\tilde{\rho}_{h_a}}{dt}=\tilde{F}_a(t,\tilde{\rho}_{h_a}(t))$
with initial value $\tilde{\rho}_{h_a}(0)=\rho_{h_a}(0)$ and

\begin{equation*}
\tilde{F}_a(t,\rho)=(\nabla h_a b_a) ( h^{-1}(\tilde{y}_a^2(t)))-(\nabla h_a b_a) (h^{-1}(\tilde{y}_a^2(t)-\rho))+ \tilde{f}_{h}(t)dt,
\end{equation*}
where
\begin{equation*}
 \tilde{f}_h(t)=-\kappa^a_1 \tilde{\rho}_{h_a}(t)-\kappa_2\frac{\tilde{\rho}_{h_a}(t)}{\sqrt{|\tilde{\rho}_{h_a}(t)|}}. 
\end{equation*}
Then, in a similar way as above  we can define $\psi(w_2)$  as the unique $w_1\in {\cal C}^{\frac{1}{2}-\varepsilon_\theta}([0,1],\ER^d)$ such that $(B_t^{w_1})_{t\in[0,1]}$ satisfies $dB_t^{w_1}=dB_t^{w_2}-\tilde{f}_h(t) dt$. By similar arguments as for $\varphi$, $\psi$ is measurable under $\mathbf{(H_0)}$. Furthermore, denoting $\tilde{y}_a^{1}(t)=\tilde{y}_a^{2}(t)-\tilde{\rho}_{h_a}(t)$, one can check that the construction ensures that if $w_2=\varphi(w_1)$ (and symmetrically if 
$w_1=\psi(w_2)$),    $\tilde{\rho}_{h_a}=\rho_{h_a}$ and $(y_a^1,y_a^{2})=(\tilde{y}_a^1,\tilde{y}_a^2)$ and it follows that $\varphi\circ\psi=\psi\circ \varphi=I_d$ on ${\cal C}^{\frac{1}{2}-\varepsilon_\theta}([0,1],\ER^d)$.

\noindent $(iv)$ Control of the function $g_h$ and Girsanov:  For {$\omega\in\Omega_{K,\alpha,\tau_{k-1}}$ and $w_1\in {\cal C}^{\theta}([0,1],\ER^d)$}, consider the explicit expression of $g_h$ given by \eqref{eq:explicitgh}.

By \eqref{hairerassumpcond}, the $L^2$-norm of $(({\cal R}_0 g_w^{\tau_{k-1}})(t))_{t\in(0,1]}$ is bounded by $1$. As concerns that of the second term in  \eqref{eq:explicitgh}, it follows from Lemma 5.1 of \cite{hairer} that it is enough to bound $f_h(0)$ and $|d f_h/dt|$.  For the sake of simplicity, we set $\rho_h=\rho_{h_a}$ with the choice of $a$ of the end of $(i)$. Since $\rho_h$ is built in such a way that $|\rho_h(t)|\le |\rho_h(0)|$,  one can check that  there exists a deterministic constant $\tilde{C}$  depending on $K$, ${M}$ and $\theta$ such that $|f_h(0)|\le \tilde{C}$ and 
$$ \forall t\in[0,1],\quad\left|\frac{d f_h(t)}{dt}\right|
\le C (\left|\rho_h(t)\right|+|\sqrt{\rho_h(t)}|)\le \tilde{C}.$$
We deduce  that  for every positive ${M}$ and $K$, there exists another finite constant ${\bar C}(M,K)$ such that 
\bqn\label{contgirsanov}
\forall \omega\in \Omega_{K,\alpha,\tau_k},\;\forall w_1\in {\cal C}^{\theta}([0,1],\ER^d) , \int_0^1 |g_h(s)|^2 ds<\bar{C}(M,K).
\eqn
 This allows us to apply Girsanov Theorem on $[0,1]$. More precisely,  denoting by $\PE_W$ the Wiener measure and by $\varphi^*\PE_W$ the image measure of $\PE_W$ by the mapping $\varphi$, one deduces from  Girsanov Theorem
that  ${\varphi}^*\PE_W(dw)= D_{{\varphi}} (w) \PE_W(dw)$ where 
\bqn\label{eq:dphi}
D_{\varphi}(w)=\exp\left(\int_0^1 ({g}_{h}(s)| dw(s))-\frac{1}{2}\int_0^1 |{g}_{h}(s)|^2ds\right).
\eqn
We can now make explicit the coupling strategy.
%
%
%

\smallskip
\noindent $(v)$ Construction of $(W^1,W^2)$ on $[\tau_{k-1},\tau_{k-1}+1]$. First, we  recall that we set $W^1_t=W^2_t$ on $[\tau_{k-1},\tau_{k-1}+1]$  if $\omega\in\Omega_{K,\alpha,\tau_k}^c$ (in this case, attempting a coupling would generate a too important cost for the future).   Now, if $\omega\in\Omega_{K,\alpha,\tau_{k-1}}$, the construction follows the lines of \cite{hairer} but with the specificity that the construction of $g_h$ leads to a successful coupling only on a subset of ${\cal C}^{\frac{1}{2}-\varepsilon_\theta}([0,1],\ER^d)$. 


More precisely, for  positive measures $\mu_1$ and $\mu_2$ with densities $D_1$ and $D_2$ with respect to another measure $\mu$, denote by $\mu_1\wedge\mu_2$ the measure defined by $(\mu_1\wedge \mu_2) (dw)=D_1(w)\wedge D_2(w) \mu(dw)$.
With the help of the function $\varphi$ introduced in $(iii)$ (see \eqref{defvarphi1}), we define a non-negative measure ${\bf P}_1$ on ${\cal C}^{\frac{1}{2}-\varepsilon_\theta}([0,1],\ER^d)^2 $ by
$${\bf {P}}_1=\frac{1}{2}\varphi_{1}^* \PE_W\wedge \varphi^*_{2}\PE_W$$
where $\varphi_{1}$ and $\varphi_{2}$ are the functions  defined on ${\cal C}^{\theta}([0,1],\ER^d)$ by
$$\varphi_1(w)=(w,\varphi(w))\quad \textnormal{and}\quad \varphi_2(w)=(\varphi^{-1}(w),w).$$
%
By construction,  
$$\varphi_{1}^* \PE_W(dw_1,dw_2)=\mathbf{1}_{\{(\varphi^{-1}(w),w)\}}(w_1,w_2)D_{{\varphi}}(w) \PE_W (dw),$$
where $D_{{\varphi}}$ is defined by \eqref{eq:dphi}. This implies that ${\bf P}_1$ satisfies
\bqn\label{eq:repp1}
{\bf P}_1(dw_1,dw_2)=\frac{1}{2}\mathbf{1}_{\{(\varphi^{-1}(w),w)\}}(w_1,w_2)(D_{\varphi}(w)\wedge 1) \PE_W (dw).
\eqn
Write  $S(w_1,w_2)=(w_2,w_1)$ and denote by $\tilde{\bf P}_1$ the ``symmetrized'' non-negative measure induced by ${\bf P}_1$,  $\tilde{\bf P}_1:={\bf P}_1+S^*{\bf P}_1$.  We then define the coupling   $(\tilde{W}^1_t,\tilde{W}^2_t)=(W^1_{t+\tau_{k-1}}-W^1_{\tau_{k-1}},W^2_{t+\tau_{k-1}}-W^2_{\tau_{k-1}})$ as follows:
$${\cal L}((\tilde{W}^1_t,\tilde{W}^2_t)_{t\in[0,1]})=\tilde{\bf P}_1+ \Delta^*(\PE_W-\Pi^*_1 \tilde{\bf P}_1)={\bf P}_1+{\bf P}_2$$
with $\Delta(w)=(w,w)$, $\Pi_1(w_1,w_2)=w_1$ and ${\bf P}_2=S^*{\bf P}_1+\Delta^*(\PE_W-\Pi^*_1 \tilde{\bf P}_1)$. 
Using \eqref{eq:repp1}, we check that for nonnegative functions $f$, 
\begin{align*}
\Pi^*_1\tilde{\bf P}_1(f)&\le\frac{1}{2}\int \left(f(\varphi^{-1}(w))D_{\varphi}(w)+ f(w)\right)\PE_{W}(dw)\leq \PE_{W}(f),
\end{align*}
hence  ${\bf P}_2$ is the sum of two positive measures.
Thanks to the symmetry property of $\tilde{\bf P}_1$ and to the fact that $ \Pi_1 \circ\Delta$ is the identity, one can also check that the marginals of $ {\bf P}_1+ {\bf P}_2$ are both equal to $\PE_W$.

\smallskip
\noindent ({\it vi}) Lower-bound for the probability of coupling: by construction, conditionally on  $\Omega_{K,\alpha,\tau_{{k-1}}}$ and  $(X^{1}_{\tau_{k-1}},X^{2}_{\tau_{k-1}}, (W^1,W^2)_{t\le \tau_{k-1}})$, under the subprobability $\ {\bf P}_1$ the coupling is successful  on the event $A_M\times\varphi(A_M)$. In other words, if we assume that $(W^1,W^2)$ is realized with the previous coupling construction, we have 
$$\PE(X_{\tau_{k-1}+1}^1=X_{\tau_{k-1}+1}^2|\Omega_{K,\alpha,\tau_{k-1}})\ge \| 1_{A_M\times\varphi(A_M)}{\bf P}_1\|_{TV}.$$
By \eqref{eq:repp1} and Lemma C.1. of \cite{mattingly02} (applied to $p=2$, $\mu_1=\varphi^*\PE_W$, $\mu_2=\PE_W$ and $X=A_{M}$) we have
$$\|\mathbf{1}_{A_M\times\varphi(A_M)}{\bf P}_1\|_{TV}\ge \frac{\left[\int_{\varphi(A_{M})} D_{\varphi}(w)\PE_W(dw)\right]^2}{ 4 \int_{\varphi(A_{M})} D_{\varphi}(w)^{3} \PE_W(dw)}.$$
We will now show that  $M$  can be chosen in such a way that the above quantity is bounded away from $0$ independently of $k\in \mathbb{N}$. 
On the one hand, by exhibiting an exponential martingale and by using \eqref{contgirsanov}, we have
\begin{align*}
\int_{\varphi(A_M)} D_{\varphi}(w)^{3} \PE_W&(dw)\le 
\Big( \sup_{w\in \varphi(A_M)}\exp(3\int_0^1 |{g}_{h}(s)|^2ds)\Big)\times\\
&\int \exp\left(3\int_0^1 (g_{h}(s)| dw_s)-\frac{3^2}{2}\int_0^1 |g_{h}(s)|^2ds\right)\PE_W(dw)\le\exp(3{\bar C}(M,K) ) .
\end{align*}
On the other hand,
$$\int_{\varphi(A_{M})} D_{\varphi}(w)\PE_W(dw)=\PE_W(A_M).$$
For   the choice   $$M:=C(K,\tilde{K}),$$ we get from $(i)$ that
$$\{w,\|w\|_{\frac{1}{2}-\varepsilon_\theta}^{0,1}\le \tilde{K}\}\subset A_{M}.$$ As a consequence, we have
$\int_{A_{M}} D_{\varphi}(w)\PE_W(dw)\geq \PE_W(\{w,\|w\|_{\frac{1}{2}-\varepsilon_\theta}^{0,1}\le \tilde{K}\})$
and 
$$\|\mathbf{1}_{A_M\times\varphi(A_M)}{\bf P}_1\|_{TV}\ge \frac{\left[  \PE_W(\{w,\|w\|_{\varepsilon_\theta}^{0,1}\le \tilde{K}\})  \right]^2}{ 4 \exp(3{\bar C}(C(K,\tilde{K}),K) )}>0$$
which concludes the proof. 

\smallskip
\noindent $(b)$ When Step 1 is successful, this property follows from \eqref{contgirsanov}. If  Step 1 is not attempted (and thus fails) since $\omega\in\Omega_{K,\alpha,\tau_{k-1}}^c$, $W^1=W^2$ on $[\tau_{k-1},\tau_{k-1}+1]$ so that $g_w$ is null on $[\tau_{k-1},\tau_{k-1}+1]$. If Step $1$ is attempted and fails,  it follows from the above construction of the coupling that $w_1=w_2$ or $w_2=\varphi^{-1}(w_1)$ with $w_1\in A_M$. Then, since the control of the functions $\tilde{f}_h$ (defined
in $(iii)$) and its derivative are similar to that of $f_h$ in $(iv)$, we deduce that the $L^2$-norm of $\tilde{g}_h(t)=g_w(t+\tau_{k-1})$ can be also bounded in a similar way.   

\smallskip
\noindent 
$(c)$ When Step $1$ is successful, $g_B(t+\tau_{k-1})=f_h(t)$ on $[0,1]$ and the boundedness of $f_h$ follows from that of $\rho_h$ which is proved in $(ii)$.

%
%
%

\end{proof}
\subsection{Step 2}
As explained before, Step 2 is a series of trials on some intervals $I_\ell$ of length $c_2 2^\ell$  (the first one  of length $2 c_2$, the second  one of length $4 c_2$,\ldots).
We denote by $s_{k,\ell}$ the  left extreme of each interval $I_\ell$. More precisely, for every $k\ge1$, we define $(s_{k,\ell})_{\ell\ge0}$ by
\bqn\label{eq:skuv}
 s_{k,0}=s_{k,1}=\tau_{k-1}+1\quad \textnormal{and for every $\ell\ge 1$} \quad s_{k,\ell+1}=s_{k,\ell}+c_2 2^{\ell}.
 \eqn
Also denote by $\ell_k^*$, the (first) trial after time  $\tau_{k-1}$ where Step 2 fails. The case $\ell_k^*=0$ and $\ell_k^*=+\infty$ correspond respectively to the failure of Step 1 and to the success of Step 2. For  given positive $\alpha$ and $K$, we set  
\bqn\label{eq:bkl}
{\cal B}_{k,\ell}:=\Omega_{K,\alpha,\tau_{k-1}}\cap\{\ell_k^*>\ell\},\quad k\ge1,\; \ell\ge0.
\eqn
With this definition, 
\bqn\label{eq:productinfty}
\PE(\tau_k=+\infty|\Omega_{K,\alpha,\tau_{k-1}})=\PE(X_{\tau_{k-1}+1}^1=X_{\tau_{k-1}+1}^2| \Omega_{K,\alpha,\tau_{k-1}}) \prod_{\ell=1}^{+\infty}\PE({\cal B}_{k,\ell}|{\cal B}_{k,\ell-1}).
\eqn	
Consequently, the aim  now is to lower-bound  $\PE({\cal B}_{k,\ell}|{\cal B}_{k,\ell-1})$. This is the purpose of Lemma \ref{lemme:step2.2}.  The proof is (once again) based on a coupling argument, which is given in the  next lemmas: 
\begin{lemme}\label{lemme:couplingmarg} Let $K$ and $b$ be positive numbers.

\smallskip
\noindent(i)  Then, there exist  $M_\bee>0$, $\rho_\bee^1$ and $\rho_\bee^2\in(0,1)$  such that for every $a\in[-\bee,\bee]$, we can build a random variable $(U_1,U_2)$ with values in $\ER^2$
such that 
\begin{align*}
 {\cal L}(U_1)={\cal L}(U_2)={\cal N}(0,1), \quad\rho_\bee^1\le \PE(U_2=U_1+a)\le \rho_\bee^2,\quad \PE(|U_2-U_1|\le M_\bee)=1
 \end{align*}
 and on the event $\{U_2=U_1+a\}$, $|U_1|\le \tilde{M}_\bee$ and $|U_2|\le \tilde{M}_\bee$ hold a.s.,  where $\tilde{M}_\bee\le \frac{M_\bee}{2}+\bee$.

\smallskip
\noindent (ii) Furthermore, if $\bee\in (0,1)$, the previous statement holds with  $M_\bee\le \max\{4\bee,-2\log(\bee/8)\}$, $\rho_1^\bee=1-\bee$ and $\rho_\bee^2=1-{\frac{\bee}{2}}$. 
\end{lemme}
\begin{Remarque}\label{rem:modiflem} In order to ensure the $(K,\alpha)$-admissibility condition at the next trials, one needs to control the increments of $W^1$ and $W^2$ during Step 2. In particular, when Step 2 fails, we will need the probability of success to be not too large. This explains the property of domination of the probability of success  $\PE(U_2=U_1+a)$ (and $\PE(W^2=W^1+g)$ in the next result) which may appear of poor interest. For the same reason, we give  in the following result  an explicit construction of $W^1$ and $W^2$ during Step 2.
\end{Remarque}
\begin{proof}
 $(ii)$ is almost the statement of Lemma 5.13 of \cite{hairer}.  The only new points are the deterministic control of $|U_1|$ and $|U_2|$ on the event $\{U_2=U_1+a\}$ and the domination of the probability of success by $\rho_\bee^2=1-\frac{\bee}{2}$. With the notations of \cite{hairer}, the first property follows from the construction of the measure ${\cal N}_3$ which is such that 
for every $a\in[-\bee,\bee]$, the support of ${\cal N}_3$ is included  in $[-M_\bee/2,M_\bee/2-a]\times[-M_\bee/2-a,M_\bee/2]$. For the second one,  it is enough to note that the probability of success introduced in Lemma 5.13 of \cite{hairer} and denoted by ${\cal N}_3(L_3)$  is a  non-decreasing continuous function of $M$ and equal to $0$ if $M_\bee=0$. Thus, the domination of this probability can be obtained by reducing sufficiently the value of $M$.

\smallskip
\noindent  On the other hand, $(i)$ is in some sense a rough version of $(ii)$. Its proof can also be done by following the lines of the lemma of \cite{hairer} and by checking that for every $\bee>0$, we can choose $M_\bee$ large enough such that 
$\inf_{a\in[-\bee,\bee]}{\cal N}_3(L_3)>0$.
\end{proof}

The following lemma is a slightly modified version of Corollary 5.14 of \cite{hairer}.
\begin{lemme}\label{cor:couplingexpansion}
Let $T$ and $\bee$ be  positive numbers and $g\in L^2([0,T],\ER)$ with $\|g\|_{2,[0,T]}\le \bee$.
\smallskip

\noindent (i) There exists $M_\bee>0$, $\rho_\bee^1$,  $\rho_\bee^2\in(0,1)$  and  a  couple of Wiener processes  $(W^1,W^2)$ defined in $[0,T]$  such that
\bqn\label{eq:propcoupl}
  \rho_\bee^1\le \PE\left(W^2_t=W^1_t+\int_0^t g(s)ds,t\in[0,T]\right)\le \rho_\bee^2\quad\textnormal{and}\quad \PE(\|W^2-W^1\|_{2,[0,T]}\le M_\bee)=1.
\eqn
 Furthermore, there exists a triple $(U_1,U_2,V)$ of  standard normally distributed random variables and a Brownian motion $\tilde{W}$ such that 
 $(U_1,U_2)$ and $(V,\tilde{W})$ are independent, 
 \bqn\label{eq:propcoupl2}
 W^i_t=\left(U^i+V\right)\frac{\int_0^t g(s)ds}{\|g\|_{{2,[0,T]}}}+\tilde{W}_t,\quad t\in[0,T] \quad \mbox{ for } i=1,2, 
 \eqn 
 and moreover  $|U_i|\le \tilde{M}_\bee:=\frac{M_\bee}{2}+\bee$ on the event $\{W^2_t=W^1_t+\int_0^t g(s)ds,t\in[0,T]\}$.
 
 \smallskip
 \noindent (ii) Furthermore, if $\bee\in(0,1)$, the previous statement holds with $M_\bee=\max\{4\bee,-2\log(\bee/8)\}$, $\rho_\bee^1=1-\bee$ and $\rho_\bee^2=1-\frac{\bee}{2}$. 
  \end{lemme}
  
\begin{proof} $(i)$  Let $(f_k)_{k\ge1}$ denote a complete orthonormal basis of $L^2([0,T],\ER)$ with $f_1= g/{\|g\|_{{2,[0,T]}}}$. In some  probability  space 
 $(\Omega',{\cal F}',\PE')$ let $(U_1,U_2)$ be a couple of random variables satisfying the properties of Lemma \ref{lemme:couplingmarg}$(i)$ and $(\xi_k)_{k\ge2}$ be a sequence independent of $(U_1,U_2)$ of i.i.d. random variables
with ${\cal L}(\xi_2)={\cal N}(0,1)$. Defining for  $i=1,2$ the process $ W^i_t: = {\cal W}^i(\mathbf{1}_{[0,t]}), t\geq 0,$
where  ${\cal W}^i: L^2([0,T],\ER)\to{\cal W}^i(L^2([0,T],\ER) )\subset  L^2(\Omega',{\cal F}',\PE')$ is the isometry of Hilbert spaces  such that ${\cal W}^i(f_1)=U^i$ and ${\cal W}^i(f_k)= \xi_k , k\geq 2$, one easily checks (by computing covariances) that
$$ W^i_t=  U_i\frac{\int_0^t g(s)ds}{\|g\|_{{2,[0,T]}}}+\sum_{k\ge2}\xi_k\int_0^t f_k(s)ds$$
and that $(W^i_t)$ is a standard Brownian Motion. It follows from   Lemma \ref{lemme:couplingmarg}$(i)$ and from the previous construction that \eqref{eq:propcoupl} holds. Furthermore, 
introducing artificially a last standard normally distributed random variable $V$ independent of $\sigma(U_1,U_2,\xi_k,k\ge2)$, we 
can write $W^i$ as follows 
$$W^i_t=(U_i+V)\frac{\int_0^t g(s)ds}{\|g\|_{{2,[0,T]}}}+\tilde{W}_t,\quad  t\in[0,T]$$
where $\tilde{W}_t:=-V\int_0^t g(s)ds/{\|g\|_{{2,[0,T]}}}+\sum_{k\ge2}\xi_k\int_0^t f_k(s)ds$ is a standard Brownian motion independent of $(U_1,U_2)$. Finally, the boundedness property of $U_i$ on $\{W^2_t=W^1_t+\int_0^t g(s)ds,t\in[0,T]\}$ again follows from that obtained in Lemma \ref{lemme:couplingmarg}$(i)$.

\smallskip
\noindent $(ii)$ The proof is identical using the properties of Lemma \ref{lemme:couplingmarg}$(ii)$ instead of those of $(i)$.
\end{proof}

Before stating the key lemma for Step 2 (below), let us introduce some notations.  Owing to the   one-to-one correspondence  between $g_w$ and $g_B$, there is  a unique  choice for function  $g_w$  in $[\tau_{k-1}+1,\infty)$ which ensures that $g_B(t)=0$ after $\tau_{k-1}+1$ (or equivalently that $B^1_t=B^2_t$ after $\tau_{k-1}+1$). We denote it by $g_{_S}$ in the next lemma (see the proof for an explicit expression of $g_{_S}$).
\begin{lemme}\label{lemme:step2.2} Let $K>0$ and assume that $\alpha\in(0,H)$. There exists a constant $C_K\ge 1$ which does not depend on $k$ such that,   
\bqne
\int_0^{+\infty} (1+t)^{2\alpha}|g_{_S}(\tau_{k-1}+1+t)|^2dt\le C_K.
\eqne
Then,  $(W_1,W_2)$  can be constructed during Step 2 in such a way that for all $k$ and $\ell$, 
\bqn\label{eq:lbbk1}
\rho_K^1\le \PE({\cal B}_{k,1}|{\cal B}_{k,0})\le\rho_K^2\;\textnormal{and $\forall \ell\ge2,$}\;(1-\rho_K^3 2^{-\alpha \ell})\le\PE({\cal B}_{k,\ell}|{\cal B}_{k,\ell-1})\le (1-\rho_K^3 2^{-\alpha \ell-1})
\eqn
where $\rho_K^1,\rho_K^2\in(0,1)$  do not depend on $k$ and $\rho_K^3=c_2^{-\alpha}\sqrt{C_K}$. In particular, if $c_2= C_K^{\frac{1}{2\alpha}}$, $\rho_K^3=1$ and in this case, if $2\le \ell_k^*<+\infty$ one has
\bqne
\int_{s_{k,\ell_k^*}}^{s_{k,\ell_k^*+1}} |g_w(t)|^2dt\le (2(\ell^*_k+3))^2\quad \textnormal{and}\quad\forall \ell \in\{2,\ldots,\ell_k^*\},\quad\int_{s_{k,\ell-1}}^{s_{k,\ell}}|g_w(t)|^2dt\le 2^{-2\alpha \ell}.
\eqne
whereas  if $\ell_k^*=1$,  one has
$\int_{s_{k,1}}^{s_{k,2}} |g_w(t)|^2dt\le C'_K$ for $C'_K$  a finite constant.
\end{lemme}
\begin{Remarque}\label{indep:K} The lower-bounds obtained in \eqref{eq:lbbk1} ensure the strict positivity of $\PE(\tau_k=+\infty|\Omega_{K,\alpha,\tau_{k-1}})$. The other properties will be needed for the sequel. Note that $c_2$ can be chosen in such a way that the involved quantities do not depend on $K$ except if $\ell=1$. 
\end{Remarque}
\begin{proof}
We  first remark that if at a positive stopping time $T_1$ one has
$X_{T_1}^1=X_{T_1}^2$, then
 $(X_t^1)$ and $(X_t^2)$ remain equal 
on $[T_1,T_2]$ (where $T_2>T_1$ is a second stopping time)  if and only if $g_B(t)=0$ on $(T_1,T_2]$. By  Lemma 4.3 in \cite{hairer}, and its proof, this condition is satisfied if and only if 
$$\forall t\in (0,T_2-T_1],\; g_w(t+T_1)= g_{_S}(t+T_1):= {\cal R}_0 g_w^{T_1}(t).$$
The interesting point is that the above function is ${\cal F}_{T_1}$-measurable (the context is thus different from Step $1$, where the function denoted by $g_h$ was defined in a   dynamic way). In particular, by conditioning on {${\cal F}_{ s_{k,\ell}}$} one can write:
$$\PE({\cal B}_{k,\ell}|{\cal B}_{k,\ell-1})=\ES( Q({\cal R}_0 g_w^{\tau_{k-1}+1}(s_{k,\ell}+.), c_2 2^\ell))$$
where for  positive $T$  and   a (deterministic) measurable function $g$ on $[0,+\infty)$ we denote
$$ Q(g, T)=\PE(\forall t\in[0,T],\;\tilde{W}^2_t=\tilde{W}^1_t+g(t)).$$
for $(\tilde{W}^1,\tilde{W}^2)$ a given  couple of Brownian motions on $[0,T]$.  By Lemma \ref{cor:couplingexpansion},  if
$\|g\|_{[0,T],2}\le b\le b_0$, we can build the couple  $(\tilde{W}_1,\tilde{W}_2)$ in such a way that  $Q(g,T)\ge (1-b)\vee \rho$ where $\rho$ depends only on $b_0$.

\smallskip
\noindent Following carefully the proof of Lemma 5.12. of \cite{hairer} (see in particular (5.36) therein), one deduces from Lemma \ref{lemma:step1} ($(b)$ and $(c)$) and Condition \eqref{hairerassumpcond} that on  ${\cal B}_{k,0}$,
$$ \int_0^{+\infty} (1+t)^{2\alpha} {\cal R}_0 g_w^{\tau_{k-1}+1}(t)dt\le C_K$$
for some positive constant $C_K$. Without loss of generality, we can assume that $C_K\ge1$. This yields the first property of the lemma and this easily implies that for every $\ell\ge 1$,
\bqne
\int_{s_{k,\ell}}^{s_{k,\ell+1}} |g_{_S}(u)|^2 du\le b_\ell^2 
\eqne
with $b_1= \sqrt{C_K}$ and $b_\ell= c_2^{-\alpha}\sqrt{C_K}2^{-\alpha \ell}$ if $\ell\ge2$. It remains to apply Lemma     \ref{cor:couplingexpansion} ($(i)$ for $\ell=1$ and $(ii)$ for $\ell\ge 2$) to obtain \eqref{eq:lbbk1}.
Finally, the bound for $\int_{s_{k,\ell_k^*}}^{s_{k,\ell_k^*+1}} |g_w(t)|^2dt$ follows from the value of $M_{b_\ell}$ given by  Lemma \ref{cor:couplingexpansion}. 

%
\end{proof}

\section{About the $(K,\alpha)$-admissibility condition}\label{sec:Kadmis}
In this section, we assume that Steps 1 and 2 are carried out  as described previously, and the aim is to ensure that  the system is $(K,\alpha)$-admissible with positive probability at all times $\tau_k$. This is the purpose of the next proposition:
\begin{prop}\label{prop:minokalp} Let  $(X_t^1,X_t^2)_{t\ge0}$ denote a solution to \eqref{eq:eqcoup} with initial condition $\tilde{\mu}$ satisfying 
$\tilde{\mu}(|x_1|^r+|x_2|^r)<+\infty$  for some  $r>0$. Assume $\mathbf{(H_0)}$, $\Hun$ and $\Hdeux$. Let $\alpha\in(0,1/2)$. Assume that for each $K>0$, $c_2$ defined in \eqref{eq:defc2} satisfies $c_2= C_K^{\frac{1}{2\alpha}}$ (where $C_K$ is a constant greater than $1$ defined in Lemma \ref{lemme:step2.2}) and that for every $k\ge1$ and $\ell\ge0$,  $\Delta_3(\ell,k)$  introduced in  \eqref{durationstep3} is defined by $\Delta_3(\ell,k)=c_3 a_k 2^{\beta \ell}$   with  $\beta>(1-2\alpha)^{-1}$, $a_k=\varsigma^k$ for some  (arbitrary) fixed  $\varsigma>1$, and $c_3$ an appropriate constant depending on the previous parameters (see Proposition \ref{lemme:step3} and Remark \ref{reducPsi}  for details).  Then, for every $\varepsilon>0$, there exists $K_\varepsilon>0$ such that for every $k\ge0$,
$$\PE(\Omega_{K_\varepsilon,\alpha,\tau_k}|\tau_{k}<+\infty)\ge 1-\varepsilon.$$ 
\end{prop}

The proof of this proposition is divided into two parts corresponding respectively to Conditions \eqref{hairerassumpcond} and \eqref{Kadmiscond}.
The first concerns the coupling function $g_w$ and the proof corresponding to this condition  easily follows from \cite{hairer} (see Subsection  \ref{subsec:hairerassum} for details).

\bigskip
{\begin{Remarque} \label{rem:valeurc2} In the sequel of this section, we always assume that $\alpha$ is a fixed number in $(0,1/2)$ and that  $c_2= C_K^{\frac{1}{2\alpha}}$. These facts are not recalled again in  each statement.
\end{Remarque}} 

The lower-bound for the second condition   is obtained in the next subsections. 
\subsection{$(K,\alpha)$-admissibility and Lyapunov}
We denote in what follows 
$${\cal E}_k:=\{\tau_k<\infty\}\;(=\{\tau_1<\infty,\ldots,\tau_k<\infty\}).$$
 We want to prove that for every $\varepsilon>0$, there exists $K_\varepsilon>0$ such that
$$\PE(\Omega_{K_\varepsilon,\tau_k}^2|{\cal E}_k)\ge 1-\varepsilon.$$
But since for every events $A_1$, $A_2$, $A_3$ and $A_4$, $\PE(\cap_{i=1}^4 A_i)\ge \sum_{i=1}^4\PE(A_i)-3$, it is 
enough to prove that for every $\varepsilon>0$, there exists $K_\varepsilon>0$ such that for $j=1,2$,
\bqn\label{eq:contmemoire}
\PE(\varphi_{\tau_k,\varepsilon_\theta}(W^j)\le K_\varepsilon|{\cal E}_k)\ge 1-\varepsilon, \quad j=1,2
\eqn 
and 
\begin{equation}\label{controlposition}
\PE(|X_{\tau_k}^{j}|\le K_\varepsilon|{\cal E}_k)\ge 1-\varepsilon\quad j=1,2.
\end{equation}
Since the arguments to prove \eqref{eq:contmemoire} are contained in those needed for the \eqref{controlposition}, we defer the proof of the former
to the appendix (see Appendix B) and we only focus on the second statement.  The proof of this property is based on a Lyapunov-type argument: owing to the Markov inequality, it is obvious that \eqref{controlposition} will be true if one exhibits a positive function $\Psi:\ER^d\rightarrow\ER$
such that $\lim_{|x|\rightarrow+\infty}\Psi(x)=+\infty$ and for which there exists a finite positive constant $C$ such that for every $k\in\EN\cup\{0\}$ and for every $K>0$,
\begin{equation}\label{controlposition2}
\ES(\Psi(X_{\tau_k}^{j})|{\cal E}_k)\le C\quad j=1,2.
\end{equation}
Note that since the construction of Step 1 depends on $K$, the independence of $C$ with respect to $K$ is primordial.
To this end, we first introduce the following contraction assumption depending on $\theta\in(1/2,H)$. 
\begin{align*}
&\Hpun:\;\textnormal{There exists a subquadratic continuous function $\Psi:\ER^d\mapsto\ER_+^*$ satisfying } \\
& \lim_{|x|\rightarrow+\infty} \Psi(x)=+\infty \textnormal{ and } \exists\,\rho\in(0,1)\;\textnormal{and}\; C>0\;\textnormal{such that}\;a.s., \forall x\in\ER^d,\\
&  \hspace{3cm} \Psi(X_1)\le \rho \Psi(x)+C(1+\|B\|_\theta^{0,1}).  \\
\end{align*}
In the previous assumption, $(X_t)_{t\ge0}$ denotes a solution to \eqref{fractionalSDE0} and \textit{subquadratic} means that there exists $C>0$ such that for every $x\in\ER^d$, $\Psi(x)\le C(1+|x|^2)$.
In Subsection \ref{subsec:lyapproof}, we will prove that, under the Lyapunov assumption $\Hun$, $\Hpun$ is true.
As detailed in the next proposition,  $\Hpun$ leads to \eqref{controlposition2} if the following condition, which will be proved in subsection \ref{subsec:hpdeux}, is also true:
\begin{flushleft}
\begin{align*}
&\Hpdeux:\;\textnormal{For every $\rho\in(0,1)$, there exists $C_\rho>0$ such that for every $k\in\EN$ and $K>0$}  \\
&\hspace{4cm}\ES[\sum_{u=1}^{\Delta \tau_k}\rho^{\Delta \tau_k  -u}\| B\|_\theta^{\tau_{k-1}+u-1,\tau_{k-1}+u}|{\cal E}_k]\le C_\rho.
\end{align*}
\end{flushleft}
\begin{prop}\label{prop:tau0} Let $\theta\in(1/2,H)$ and assume $\Hpun$. Let  $(X_t^1,X_t^2)_{t\ge0}$ denote pair  of solutions to \eqref{eq:eqcoup} with initial condition $\tilde{\mu}$ satisfying 
$\tilde{\mu}(\Psi^2(x_1)+\Psi^2(x_2))<+\infty$. For  $x_1,x_2\in\ER^d$, set  
\bqn\label{def:tau0}
\tau_0(x_1,x_2):=\inf\{u\in\EN, \rho^u (\Psi(x_1)+\Psi(x_2))\le 1\}.
\eqn
Assume that $(\tau_k)_{k\ge1}$ is built in such a way that, $\Hpdeux$ holds, that for every $k\ge1$,
 $\PE({\cal E}_k|{\cal E}_{k-1})\ge \delta_1>0$ (where $\delta_1$ is a positive number which does not depend on $k$) and that  $\Delta \tau_k\ge \frac{\log(\delta_1/2)}{\log\rho}$. Then, there exists a positive constant $C$ such that  for every $k\in\EN\cup\{0\}$ and $K>0$,
\begin{equation*}
\ES(\Psi(X_{\tau_k}^{j})|{\cal E}_k)\le C,\qquad j=1,2.
\end{equation*}
\end{prop}
\begin{Remarque}\label{reducPsi} $\rhd$ Under the assumptions of Proposition \ref{prop:minokalp}, $\Delta \tau_k\ge c_3$ (see \eqref{durationstep3}). To ensure that $\Delta \tau_k\ge \frac{\log(\delta_1/2)}{\log\rho}$, one can thus choose $c_3$ large enough in order that $c_3\ge   \frac{\log(\delta_1/2)}{\log\rho}$.\smallskip

\noindent $\rhd$ By the elementary inequalities $|u+v|^p\le |u|^p+|v|^p$ and $|u|^p\le C(1+|u|)$ for $p\in(0,1)$, one remarks that if  $\Hpun$ holds for $\Psi$, it also holds for $\Psi^p$ if $p<1$. Since $\Psi$ is subquadratic, it follows that one can assume without loss of generality that 
$\Psi^2(x)\le C(1+|x|^r)$ for some given $r>0$. This explains the assumption $\tilde{\mu}(|x_1|^r+|x_2|^r)<+\infty$ in Proposition \ref{prop:minokalp}.
\end{Remarque}
\begin{proof} By $\Hpun$ and an induction  
$$\Psi(X_{\tau_k}^{j})\le \rho^{\Delta \tau_k} \Psi(X_{\tau_{k-1}}^{j})+C\sum_{\ell=1}^{\Delta \tau_k}\rho^{\Delta \tau_k  -\ell}(1+\| B\|_\theta^{\tau_{k-1}+\ell-1,\tau_{k-1}+\ell}).$$
First, since $\Delta \tau_k\ge \frac{\log ({\delta_1}/2)}{\log \rho}$, we deduce that $\rho^{\Delta \tau_k}\le \frac{{\delta_1}}{2}$. Thus,
$$\ES[\Psi(X_{\tau_k}^{j})|{\cal E}_k]\le \frac{{\delta_1}}{2} \ES[\Psi(X_{\tau_{k-1}}^{j})|{\cal E}_k]+C\sum_{u=0}^{+\infty}\rho^{u}+C\ES[\sum_{\ell=1}^{\Delta \tau_k}\rho^{\Delta \tau_k  -\ell}\| B\|_\theta^{\tau_{k-1}+\ell-1,\tau_{k-1}+\ell}|{\cal E}_k].$$
Since ${\cal E}_k\subset {\cal E}_{k-1}$ and $\PE({\cal E}_k|{\cal E}_{k-1})\ge{\delta_1}$,  $\ES[\Psi(X_{\tau_{k-1}}^{j})|{\cal E}_k]\le \delta_1^{-1}\ES[\Psi(X_{\tau_{k-1}}^{j})|{{\cal E}_{k-1}}]$. It follows that

$$\ES[\Psi(X_{\tau_k}^{j})|{\cal E}_k]\le \frac{1}{2} \ES[\Psi(X_{\tau_{k-1}}^{j})|{\cal E}_{k-1}]+\frac{C}{1-\rho}+C\ES[\sum_{\ell=1}^{\Delta \tau_k}\rho^{\Delta \tau_k  -\ell}\| B\|_\theta^{\tau_{k-1}+\ell-1,\tau_{k-1}+\ell}|{\cal E}_k].$$
Assumption  $\Hpdeux$ combined with  an induction then yields
$$\sup_{k\ge0}\ES[\Psi(X_{\tau_k}^{j})|{\cal E}_k]\le \ES[\Psi(X_{\tau_0}^{j})|{\cal E}_0]+\tilde{C}_\rho$$
where $\tilde{C}_\rho$ neither depends on $k$ and  $j$ nor on the starting condition $\tilde{\mu}$. Noticing that ${\cal E}_0=\Omega$,  it remains to bound
$\ES[\Psi(X_{\tau_0}^{j})]$.  By the definition of $\tau_0$ (which is ${\cal F}_0$-measurable) and the Cauchy-Schwarz inequality,
\bqn\label{dizepoizeo}
\ES_{\tilde{\mu}}[\Psi(X_{\tau_0}^{j})]\le \sum_{u=0}^{+\infty}\ES_{\tilde{\mu}}[\Psi^2(X_u^{j})]^{\frac{1}{2}}(\tilde{\mu}(\tau_0=u))^{\frac{1}{2}}.
\eqn
On the one hand, checking that for $\varepsilon>0$, there exists $C_\varepsilon>0$ such that for all $u$, $v$ of $\ER^d$,  $|u+v|^2\le (1+\varepsilon)|u|^2+C_\varepsilon|v|^2$, one deduces from $\Hpun$ that there exists $0<\tilde{\rho}<1$ and $C_{\tilde{\rho}}$ such that for every starting point $x$, 
$$\Psi^2(X_1)\le \tilde{\rho} \Psi^2(x)+C_{\tilde{\rho}}(1+\|B\|_\theta^{0,1})^2.$$ 
Thus, it again follows from an induction and from the stationarity of the increments of the fBm that 
$$\ES_{\tilde{\mu}}[\Psi^2(X_u^{j})]\le \int \Psi^2(x_j)\bar{\mu}_j(dx_j)+\frac{C_{\tilde{\rho}}}{1-\tilde{\rho}}\ES[(1+\|B\|_\theta^{0,1})^2]<+\infty,$$
since $\int \Psi^2(x_j)\bar\mu_j(dx_j)<+\infty$. It remains to control the queue of $\tau_0$. We have
\bqn\label{eq:borntau0}
\tilde{\mu}(\tau_0\ge u)\le \sum_{j=1}^2\bar\mu_j(\rho^u\Psi(x_j)>\frac{1}{2})\le 2\sum_{j=1}^2 \rho^{ u}\int\Psi (x_j)\bar\mu_j(dx_j)\le C\rho^u.
\eqn
Plugging the previous inequality in  yields the boundedness of $\ES_{\tilde{\mu}}[\Psi(X_{\tau_0}^{j})]$.
\end{proof}
As a consequence,  it remains now to prove $\Hpun$ and $\Hpdeux$. This is the purpose of the next subsections.
\subsection{Proof of $\Hpun$}\label{subsec:lyapproof}
\begin{prop}\label{prop:lyapounov} Assume $\Hun$. Then, $\Hpun$ holds for every $\theta\in(\frac{1}{2},H)$ with $\Psi=V^{\frac{2\theta-1}{4}}$.
\end{prop}
\begin{proof}
The proof is divided in four steps. In all of them, we assume that  $0\le s<t\le 1$. 

\smallskip
\noindent 
\textbf{Step 1}. We prove the following statement: there exists $C>0$ such that 
\bqn\label{eq:grown}
|X_t|\le C \left(|X_s|+C(t-s)+|\int_{s}^t \sigma(X_u) dB_u|\right)\quad a.s.
\eqn
Actually, using that $b$ is a sublinear function,
$$|X_t|\le |X_s|+C(t-s)+|\int_{s}^t \sigma(X_u) dB_u|+\int_s^t|X_u| du.$$
The result then follows from Gronwall lemma (note that the time-dependence of the Gronwall constant does not appear since $s,t\in[0,1]$).

\smallskip
\noindent\textbf{Step 2}. \textit{Control of the Hölder norm of $X$ in a small  (random) interval}: Let $\theta\in (1/2,H)$. We show that there exist some positive constants $c_0$ and $C$ such that for every $0\le s<t\le 1$, 
satisfying $c_0(1+\|B\|_{\theta}^{0,1})(t-s)^{\theta}\le \frac{1}{2}$,
\bqn\label{eq:normh}
\|X\|_{\theta}^{s,t}\le {C}\left(\|B\|_{\theta}^{0,1}+(1+|X_s|)(t-s)^{1-\theta}\right).
\eqn

Let us prove this property.  Owing to  the classical controls of Young integrals (see $e.g.$ \cite{young}, Inequality (10.9)), for every $(p,q)\in(0,1]^2$ with $p+q>1$, there exists $C_{p,q}>0$  such that for every $p$-Hölder and $q$-Hölder functions $f$ and $g$ (respectively), 
for every $0\le s< t \le 1$,  
\bqn\label{eq:youngi}
|\int_s^t f(u) dg(u)-f(s)(g(t)-g(s))|\le C_{p,q}\|f\|_p^{0,1}\|g\|_q^{0,1} (t-s)^{p+q}.
\eqn
Applying the previous inequality with $p=q=\theta$ and using that $\sigma$ is Lipschitz continuous and bounded, we deduce  that for every $0\le s\le u<v\le t \le 1$,  
\begin{align}
|\int_u^v \sigma(X_w) dB_w|&\le C\|B\|_{\theta}^{u,v}(v-u)^{\theta}\left(\|X\|_{\theta}^{u,v}(v-u)^{\theta}+\|\sigma\|_\infty\right)\nonumber\\
&\le C\|B\|_{\theta}^{0,1}(v-u)^{\theta}\left(\|X\|_{\theta}^{s,t}(t-s)^{\theta}+\|\sigma\|_\infty\right)\label{eq:d23}
\end{align}
By \eqref{eq:grown} and what precedes, we also have
$$\int_{u}^v |b(X_r)|dr\le C(v-u)\left(1+|X_s|+\|B\|_{\theta}^{0,1}(t-s)^{\theta}(1+\|X\|_{\theta}^{s,t}(t-s)^{\theta})\right).$$
Using the previous inequalities, we deduce that
$$\|X\|_{\theta}^{s,t}\le C\left(\|B\|_{\theta}^{0,1}+(1+|X_s|)(t-s)^{1-\theta}\right)+ C\|X\|_{\theta}^{s,t}\|B\|_\theta^{0,1}(t-s)^{\theta}$$
and \eqref{eq:normh} follows.

\smallskip
\noindent

\textbf{Step 3}. \textit{Control of $\sup |X_u|$ in a small (random) interval}: let $\theta\in(1/2,H)$.
 There exists $C>0$ such that for every $0\le s\le t\le 1$ satisfying 
 $c_0(1+\|B\|_{\theta})(t-s)^{\theta}\le \frac{1}{2}$ 
\bqn\label{eq:controlsup}
\sup_{s\le u\le t} |X_u|\le C(1+|X_s|).
\eqn
Actually, using that $\|B\|_{\theta}^{0,1}(t-s)^{\theta}\le (2 c_0)^{-1}$, we deduce from \eqref{eq:d23} that  for every $0\le s\le t\le 1$ satisfying 
  $c_0(1+\|B\|_{\theta}^{0,1})(t-s)^{\theta}\le \frac{1}{2}$ 
\begin{equation*}
|\int_s^t \sigma(X_v) dB_v|\le C\left(\|X\|_{\theta}^{s,t}(t-s)^{\theta}+1\right).
\end{equation*}
Using again that $\|B\|_{\theta}^{0,1}(t-s)^{\theta}\le (2 c_0)^{-1}$, it follows from \eqref{eq:normh} that
\bqne
|\int_s^t \sigma(X_v) B_v|\le C\left(1+(1+|X_s|)(t-s)\right).
\eqne
Then, it is enough to plug this control in \eqref{eq:grown} to obtain \eqref{eq:controlsup}.

\smallskip
\noindent
\textbf{Step 4}. \textit{Use of the Lyapunov assumption}. Let $V$ be such that Assumption $\Hun$ holds. Let $\theta\in (1/2,H)$. Then, there exists $\bar\rho\in(0,1)$ and $C>0$ such that for every $x\in\ER^d$,
\bqn\label{eq:lyap1}
 V(X_1)\le \bar\rho V(x)+C(1+\|B\|_\theta)^{\frac{4}{2\theta-{1}}}.
 \eqn
Let us prove this statement. By $e.g.$ \cite{zahle} (see Theorem 4.3.1) and  Assumption $\Hun$, 
\bqn\label{eq:ito}
\begin{split}
e^{\kappa_0 (t-s)} V(X_t)&= V(X_s)+\int_s^t e^{\kappa_0 (u-s)}\left((\nabla V|b)(X_u)
+ \kappa_0 V(X_u)\right)du+\int_s^te^{\kappa_0 (u-s)}(\nabla V (X_u)|\sigma(X_u) dB_u)\\
&\le V(X_s)+\beta_0(t-s) +\int_s^te^{\kappa_0 (u-s)}(\nabla V (X_u)|\sigma(X_u) dB_u).
\end{split}
\eqn
Using that the functions $\nabla V$ and $\sigma$ are Lipschitz continuous, that $\sigma$ is bounded and that $u\mapsto e^{\kappa_0 u}$ is bounded and Lipschitz continuous on $[0,1]$, we obtain that for every $0\le u<v\le 1$,
$$ |e^{\kappa_0 (v-s)} \nabla V\sigma(X_v)-e^{\kappa_0 (u-s)} \nabla V\sigma(X_u)|\le C\left((1 +|\nabla V(X_v)|)(|X_v-X_u|)+|\nabla V(X_v)|(v-u)\right).$$
By \eqref{eq:youngi}, it follows that
\begin{align*}
|\int_s^t e^{\kappa_0 (u-s)}(\nabla V (X_u)|\sigma(X_u) dB_u)|&\le C \Big ((1 +\sup_{v\in[s,t]}|\nabla V(X_v)|)(\|X\|_\theta^{s,t}+(t-s)^{1-{\theta}})(t-s)^{\theta}\\
&+|\nabla V(X_s)| \Big)\|B\|_\theta^{s,t}(t-s)^{\theta}.
\end{align*}
From now on, assume that $(1+\|B\|_{\theta}^{0,1} (t-s)^\theta)\le (2 c_0)^{-1}$. By \eqref{eq:controlsup} and the fact that $1+|\nabla V(x)|\le C_1(1+|x|)\le C_2\sqrt{V}(x)$, we have
$$1 +\sup_{v\in[s,t]}|\nabla V(X_v)|\le C \sqrt{V(X_s)}.$$
Owing  to \eqref{eq:normh} and to some reductions implied by the previous inequality, we obtain
$$|\int_s^t e^{\kappa_0 (u-s)}(\nabla V (X_u)|\sigma(X_u) dB_u)|\le C \left(V(X_s)\|B\|_\theta^{0,1}(t-s)^{1+\theta}+
\sqrt{V}(X_s)\|B\|_{\theta}^{0,1}(t-s)^{\theta}\right).$$
Set $\tilde{\theta}:=\frac{1}{2}(\theta-\frac{1}{2})$ (so that $2(\theta-\tilde{\theta})=\frac{1}{2}+\theta$). By the inequality $|x y|\le \frac{1}{2}(|x|^2+|y|^2),$
$$\sqrt{V}(X_s)\|B\|_{\theta}^{0,1}(t-s)^{\theta}\le \frac{1}{2}\left( (t-s)^{2(\theta-\tilde{\theta})} V(X_s)+(\|B\|_{\theta}^{0,1})^2(t-s)^{2\tilde{\theta}}\right)$$
and on the other hand,
$$V(X_s) \|B\|_{\theta}^{0,1} (t-s)^{1+\theta}\le (t-s)^{2(\theta-\tilde{\theta})} V(X_s) \|B\|_{\theta}^{0,1} (t-s)^{\frac{1}{2}}.$$
Now, we set  
\bqn\label{eq:eta}
\eta= (2 c_0(1+\|B\|_\theta^{0,1}))^{-\frac{1}{{\theta}}}\wedge (1+\|B\|_{\theta}^{0,1})^{-\frac{1}{\tilde{\theta}}} 
\eqn
in order that for every $0\le s<t\le 1$ such that $t-s\le \eta$,
$$c_0(1+\|B\|_{\theta}^{0,1})\eta^{{\theta}}\le \frac{1}{2}\quad\textnormal{ and} \quad (\|B\|_{\theta}^{0,1}))^2(t-s)^{2\tilde{\theta}}\le 1.$$
For such $s,t$, we finally obtain (using that $1/2\ge \tilde{\theta}$ and that $2(\theta-\tilde{\theta})=\frac{1}{2}+\theta$),
\bqn\label{eq:contint}
|\int_s^t e^{\kappa_0 (u-s)}(\nabla V (X_u)|\sigma(X_u) dB_u)|\le C (t-s)^{\frac{1}{2}+\theta} V(X_s)+ \tilde{\beta}
\eqn
where $\tilde{\beta}$ is a positive constant. Plugging this control into \eqref{eq:ito}, we deduce: for every $0\le s<t\le 1$ such that $t-s\le \eta$,
$$ V(X_t)\le e^{-\kappa_0 (t-s)} V(X_s) (1+C (t-s)^{\frac{1}{2}+\theta})+\hat{\beta}$$
where $\hat{\beta}=\beta_0 \eta+ \tilde{\beta}$.
Using that  $e^{-\kappa_0 u}\le 1-\kappa_0 u+ \frac{(\kappa_0 u )^2}{2}$ in a right neighborhood of $0$ and that $\frac{1}{2}+\theta>1$,  we can find $u_0\in[0,1]$ (depending on $\kappa_0$, $\theta$ and $C$) such that 
$$\forall u\in[0,u_0], \quad e^{-\kappa_0 u}(1+Cu^{\frac{1}{2}+\theta})\le 1-\frac{\kappa_0}{2}u.$$
Thus, for every $0\le s<t\le 1$ such that $t-s\le \tilde{\eta}:=\eta\wedge u_0$,
$$ V(X_t)\le (1-\frac{\kappa_0}{2}(t-s)) V(X_s)+\hat{\beta}.$$
In particular, applying this control on $[k\tilde{\eta}, ((k+1)\tilde{\eta})\wedge 1]$ for $k\in\{0,\ldots,\lfloor \frac{1}{\tilde{\eta}}\rfloor\}$ yields
$$ V(X_1)\le (1-\frac{\kappa_0}{2}\tilde{\eta})^{\lfloor \frac{1}{\tilde{\eta}}\rfloor} V(x)+\sum_{k=1}^{\lfloor \frac{1}{\tilde{\eta}}\rfloor} (1-\frac{\kappa_0}{2}\tilde{\eta})^{{\lfloor \frac{1}{\tilde{\eta}}\rfloor}-k} \hat{\beta}.$$
It follows from standard computations that  
$$ V(X_1)\le \exp(-\frac{\kappa_0}{2}+\tilde{\eta}) V(x)+\frac{2\tilde{\beta}}{\kappa_0\tilde{\eta}}.$$
We can assume without loss of generality that $u_0\le \kappa_0/4$ so that 
$$\exp(-\frac{\kappa_0}{2}+\tilde{\eta})\le e^{-\frac{\kappa_0}{4}}=:\bar\rho.$$
Finally, since $2/\tilde{\theta}\ge 1/\theta$, one can check that there exists $C>0$ such that  
$$\tilde\eta^{-1}\le C(1+\|B\|_\theta^{0,1})^{\frac{1}{\tilde{\theta}}}.$$ 
Since $2/\tilde{\theta}=\frac{4}{2\theta-{1}}$, this concludes the proof of Step 4. 

\smallskip 
\noindent To prove the proposition, it remains now to set $\Psi=V^{{\tilde{\theta}}}$ and to apply the inequality $|u+v|^{\bar{p}}\le |u|^{\bar{p}}+|v|^{\bar{p}}$ (which holds for every real numbers $u,v$ and $\bar{p}\in(0,1]$) with $\bar{p}=\tilde{\theta}$.
\end{proof}

\subsection{Proof of $\Hpdeux$}\label{subsec:hpdeux}
The main result of this section is Proposition \ref{prop:contdubruit}. Before, we need to establish several lemmas related to the control of the past of the fBm.\smallskip

\noindent Let $j\in\{1,2\}$. We recall that for every $0\le s<t$,
$$B_t^j-B_s^j=\alpha_H\left(\int_{-\infty}^s (t-r)^{H-\frac{1}{2}}-(s-r)^{H-\frac{1}{2}} dW_r^j+\int_s^t (t-r)^{H-\frac{1}{2}} dW_r^j\right).$$
This can be rewritten
$$B_t^j-B_s^j=\alpha_H\left(\int_{-\infty}^{\lfloor s\rfloor -1} (t-r)^{H-\frac{1}{2}}-(s-r)^{H-\frac{1}{2}} dW_r^j+
\Gamma_1(s,t,W^j)-\Gamma_2(s,t,W^j)+\Gamma_3(s,t,W^j)\right)$$
where, setting $h=t-s$, 
\begin{align*}
&\Gamma_1(s,t,W^j)=\int_{\lfloor s\rfloor -1}^{s-h} (t-r)^{H-\frac{1}{2}}-(s-r)^{H-\frac{1}{2}} dW_r^j,\\
&\Gamma_2(s,t,W^j)=\int_{s-h}^{s} (s-r)^{H-\frac{1}{2}} dW_r^j\quad\textnormal\quad \Gamma_3(s,t,W^j)=\int_{s-h}^{t} (t-r)^{H-\frac{1}{2}} dW_r^j.
\end{align*}
Let $k\ge 1$. Assume that $\tau_{k-1}<+\infty$ and that $\tau_{k-1}\le s<t\le \lfloor s\rfloor+1$. Setting $\tau_{-1}=-\infty$, we  choose to decompose the first right-hand member with respect to the sequence $(\tau_k)_{k\ge-1}$:
$$\int_{-\infty}^{\lfloor s\rfloor -1} (t-r)^{H-\frac{1}{2}}-(s-r)^{H-\frac{1}{2}} dW_r^j=\sum_{m=0}^{k}\Lambda_{m,k}(s,t,W^j)$$
with 
$$\Lambda_{m,k}(s,t,W^j)=\begin{cases} \int_{\tau_{m-1}}^{\tau_m\wedge \tau_{k-1}-1} (t-r)^{H-\frac{1}{2}}-(s-r)^{H-\frac{1}{2}} dW_r^j&\textnormal{if $m\in\{0,\ldots, k-1\}$}\\
&\\
\int_{\tau_{k-1}-1}^{\lfloor s\rfloor -1}(t-r)^{H-\frac{1}{2}}-(s-r)^{H-\frac{1}{2}} dW_r^j&\textnormal{if $m=k$.}
\end{cases}
$$
Note that for $i=1,2,3$, $\Gamma_i$ is related to the local behavior of the fBm whereas for $m=0,\ldots, k$, $\Lambda_{m,k}$ is a memory term.
The idea of the sequel of the proof is to bound $\|B\|_{\theta}^{u,u+1}$ $(u\in\{\tau_{k-1},\ldots,\tau_k\})$ through the study of the $\Gamma_i$ and the $\Lambda_{m,k}$.
With a slight abuse of notation, we will sometimes write 
\begin{equation}\label{eq:notabus}
\|\Gamma_i^j\|_\theta^{a,b}=\sup_{a\le s<t\le b} \frac{|\Gamma_i(s,t,W^j)|}{|t-s|^{\theta}}\quad \textnormal{and}\quad \|\Lambda_{m,k}\|_\theta^{a,b}=\sup_{a\le s<t\le b} \frac{|\Lambda_{m,k}(s,t,W^j)|}{|t-s|^{\theta}}.
\end{equation}
The starting point of the study of the $\Lambda_{m,k}$ is  the following lemma:
\begin{lemme}\label{lemme:IPP} Let $a<b<s<t.$ Let $W$ be a two-sided Brownian motion. Then, 
$$\frac{1}{t-s}\left|\int_a^b (t-r)^{H-\frac{1}{2}}-(s-r)^{H-\frac{1}{2}} dW_r\right|\le (t-a)^{H-\frac{3}{2}}|W_b-W_a|+\frac{1}{2}\int_a^b (s-r)^{H-\frac{5}{2}}|W_r-W_b|dr.$$
\end{lemme}
\begin{proof}
By an integration by parts,
\begin{equation}\label{eq:ipp11}
\begin{split}
\int_a^b (t-r)^{H-\frac{1}{2}}-(s-r)^{H-\frac{1}{2}} dW_r&=\left((t-a)^{H-\frac{1}{2}}-(s-a)^{H-\frac{1}{2}}\right)(W_{a}-W_b)\\
&+(H-\frac{1}{2}) \int_{a}^{b} \left((t-r)^{H-\frac{3}{2}}-(s-r)^{H-\frac{3}{2}}\right)(W_r-W_b)dr.
\end{split}
\end{equation}
On the one hand, by the elementary inequality $(1+x)^\rho\ge 1+x$ for every $x\in(-1,0]$ and $\rho\in(0,1]$, we  remark that
\begin{align*}
0\le (t-a)^{H-\frac{1}{2}}-(s-a)^{H-\frac{1}{2}}&=
(t-a)^{H-\frac{1}{2}}\left(1-\left(1+\frac{s-t}{t-a}\right)^{H-\frac{1}{2}}\right)\\
&\le (t-s)(t-a)^{H-\frac{3}{2}}
\end{align*}
On the other hand, by the inequality $(1+x)^\rho\ge1+\rho x$ for $x\ge0$ and $\rho<0$, we obtain similarly 
\begin{align*}
(s-r)^{H-\frac{3}{2}}-&(t-r)^{H-\frac{3}{2}}\le \left(\frac{3}{2}-H\right)(t-s) (s-r)^{H-\frac{5}{2}}.
\end{align*}
The result follows (using that $(3/2-H)(H-1/2)\le 1/2$).
\end{proof}
In the next lemma, we propose to bound some quantities which are related to those which appear in the previous lemma on some sub-intervals of $[\tau_{m-1},\tau_m]$ where $m\in\EN$. With the notations introduced in \eqref{durationstep3} and in \eqref{eq:skuv}, we set 
$$\tau_{m}^0=\tau_{m-1},\quad\tau_{m}^1=\tau_{m-1}+1+ 2 c_2,\quad \tau_{m}^2=s_{m,\ell_m^*}\vee \tau_{m}^1,\quad \tau_{m}^3=s_{m,\ell_m^*+1}\quad \textnormal{and} \quad \tau_{m}^4=\tau_m.$$ Since $c_3$ defined in \eqref{durationstep3} satisfies $c_3\ge 2 c_2$, Step 3 is longer than $ 2 c_2$ and $i\mapsto\tau_m^i$ is non-decreasing. Furthermore, $\tau_{m}^0$ is the beginning of Step 1, $\tau_{m}^1$ denotes the end of the first trial of Step 2 (or some time during Step 3)  if Step 1 is successful (resp. if  Step 1 fails). If Step 1 and the first trial of Step 2 are successful, $\tau_{m}^2$ and $\tau_{m}^3$ correspond to the beginning and to the end of the failed trial of Step 2. If $\ell_m^*\in\{0,1\}$, $\tau_m^1=\tau_m^2=\tau_m^3$.

 Note that $\tau_m^1$ is defined as the end of the first trial of Step $2$, instead of the end of Step 1 as it could be expected.  
 Without going into the technical details, let us remark that this particular cutting of the interval is due to the dependence in $K$ (which appears in the $(K,\alpha)$-admissibility condition) of the probability of success of the first trial of Step 2 (and that this dependence does not appear for the next trials, see Remark \ref{indep:K} for background) and that, in view of Assumption $\Hpdeux$, it is of first importance that the next results be obtained independently of $K$.
 
\begin{lemme}\label{lemma:lembrown} Assume that there exists ${\delta_1}>0$ such that for all $m\in\EN$ and $K>0$ $\PE({\cal E}_{m+1}|{\cal E}_{m})\ge {\delta_1}>0$. Then,
for every $p\ge 1$ and $\varepsilon\in(0,1)$, there exists $C_{p,\varepsilon,{\delta_1}}\in\ER_+^*$ such that for every $m\in\EN$, $i\in\{0,\ldots,3\}$, $j\in\{1,2\}$ and $K>0$,

\smallskip
\noindent (i)
\bqn\label{eq:lembrown2}
\ES\left[\left(\int_{\tau_{m}^i}^{\tau_{m}^{i+1}}\left|(1+\tau_{m}^{i+1}-r)^{-\left(\frac{3}{2}+\varepsilon\right)}(W_{\tau_{m}^{i+1}}^j-W_r^j)\right| dr\right)^p|{\cal E}_m\right]\le C_{p,\varepsilon,{\delta_1}}.
\eqn
(ii) If $\tau_m^i\neq\tau_{m}^{i+1}$,
\bqn\label{eq:lembrown}
\ES\left[\left|(\tau_{m}^{i+1}-\tau_{m}^i)^{-\left(\frac{1}{2}+\varepsilon\right)}\left(W_{\tau_{m}^{i+1}}^j-W_{\tau_{m}^i}^j\right)\right|^p|{\cal E}_m\right]\le C_{p,\varepsilon,{\delta_1}},
\eqn
\end{lemme}
\begin{Remarque} The proof of this lemma could be shortened by using some rougher arguments similar to those of the proof of Proposition \ref{prop:contdubruit}  below (see \eqref{argumentrough}). However, the  arguments given here do provide an understanding of what implies the conditioning by $\{\tau_m<+\infty\}$, or in other words, to how the distribution of the Wiener process is deformed  by the coupling attempt.  To this end and when it is possible (especially in the case $i=1$), we thus choose an approach by which we try to make explicit these distortions. 
\end{Remarque}
\begin{proof} 
%
$(i)$ By a change of variable, for every $i\in\{0,1,2,3\}$, 
$$\int_{\tau_{m}^i}^{\tau_{m}^{i+1}}(1+\tau_{m}^{i+1}-r)^{-\left(\frac{3}{2}+\varepsilon\right)}|W_{\tau_{m}^{i+1}}^j-W_r^j| dr=H_i(\tau_{m}^{i+1}-\tau_{m}^i)$$
where for a given $c>0$,
$$H_i(c)=\int_0^c (1+u)^{-\frac{1}{2}} |W^j_{\tau_{m}^{i+1}}-W^j_{\tau_{m}^{i+1}-u}| \nu(du)\quad \textnormal{with}\quad \nu(dr)=(1+u)^{-1-\varepsilon} du.$$
Noticing  that $\nu([0,c])\le \varepsilon^{-1}$, we deduce from  Jensen inequality that for every $p\ge1$,
\bqn\label{eq:HMC}
(H_i(c))^p\le \left(\frac{1}{\varepsilon}\right)^{p-1}\int_0^c (1+u)^{-\frac{p}{2}-1-\varepsilon}|W^j_{\tau_{m}^{i+1}}-W^j_{\tau_{m}^{i+1}-u}|^{p}du.
\eqn
Now, we focus successively on  cases $i=0,1,2,3$:

\smallskip
\noindent \underline{$i=0$:}  In this case, $\tau_{m}^{1}-\tau_{m}^{0}$ is deterministic and is equal to $\bar{c}:=1+2 c_2$.  
Using that ${\cal E}_{m}\subset{\cal E}_{m-1}$ and the Cauchy-Schwarz inequality, we have for every $u\in[0,\bar{c}]$ and $m\ge1$,
$$\ES[|W_{\tau_{m}^{1}}-W_{\tau_{m}^{1}-u}|^{{p}}|{\cal E}_m]\le  \frac{\ES[|W^j_{\tau_{m}^{1}}-W^j_{\tau_{m}^{1}-u}|^{2p}|{\cal E}_{m-1}]^\frac{1}{2}}{\PE({\cal E}_m|{\cal E}_{m-1})^{\frac{1}{2}}}.$$
But, conditionally on $\{\tau_{m-1}<+\infty\}$, $(W^j_{\tau_{m-1}+u}-W^j_{\tau_{m-1}},u\ge0)$ is a Brownian motion independent of ${\tau_{m-1}}$ so that 
$$\ES[|W^j_{\tau_{m}^{i+1}}-W^j_{\tau_{m}^{i+1}-u}|^{2p}|{\cal E}_{m-1}]=u^{p}.$$
Then, since $\PE({\cal E}_m|{\cal E}_{m-1})\ge \delta_1$, one deduces that 
\begin{equation}\label{eq:bro0}
\sup_{u\in [0,\bar{c}]} u^{-\frac{p}{2}} \ES[|W^j_{\tau_{m}^{1}}-W^j_{\tau_{m}^{1}-u}|^{{p}}|{\cal E}_m]\le \delta_1^{-\frac{1}{2}}.
\end{equation} 
Plugging this control into \eqref{eq:HMC} yields the result when $i=0$ with $C_{p,\varepsilon,\delta_1}=\delta_1^{-\frac{1}{2}}\varepsilon^{-p}$.

\smallskip
\noindent \underline{$i=1$:} If $\ell_k^*\in\{0,1,2\}$, $\tau_{m}^1=\tau_{m}^2$. Otherwise, we first write ${\cal E}_m=\bigcup{\cal A}_{m,\ell}$ where  ${\cal A}_{m,\ell}={\cal B}_{m,\ell}^c\cap {\cal B}_{m,\ell-1}$.  
We recall that ${\cal A}_{m,0}$ corresponds to the failure of Step 1 and for every $\ell\geq 1$, 
$ {\cal A}_{m,\ell}$ is the event that Step 2 failed after exactly $\ell$ trials. 

With the notations introduced in \eqref{eq:skuv}, we recall that on  ${\cal A}_{m,\ell}$, $\tau_m^1=s_{m,2}$ and 
$\tau_{m}^2=s_{m,\ell}$.
By \eqref{eq:HMC}, it is enough to show that  for every $\ell\ge3$,
\bqn\label{eq:etape2simp}
\begin{split}
\int_{0}^{s_{m,\ell}-s_{m,2}}(1+u)^{-\frac{p}{2}-1-{\varepsilon}}
\ES[|(W_{s_{m,\ell}}^j-W_{s_{m,\ell}-u}^j)|^p|{\cal A}_{m,\ell}] du\le C_{p,\varepsilon,\delta_1}.
\end{split}
\eqn
where $ C_{p,\varepsilon,\delta_1}$ does not depend on $k$, $m$, $\ell$ and $K$. 
With the notations of  Lemma  \ref{lemme:step2.2}, we know that on the event ${\cal A}_{m,\ell}$ with $\ell \ge 2$, we have for all $\kbis\in\{1,\ldots,\ell-1\}$,
$$ \forall t\in[s_{m,\kbis},s_{m,\kbis+1}], \quad W^2_t=W^1_t+\int_{s_{m,\kbis}}^t g_{_S}(s)ds\quad$$
where $g_{_S}$ is a ${\cal F}_{\tau_{m-1}+1}$-measurable function (defined in Lemma \ref{lemme:step2.2}). 
Moreover, by Lemma \ref{cor:couplingexpansion}$(ii)$, which can be applied with $b= 2^{-\alpha v}$ (owing to Lemma \ref{lemme:step2.2} and Remark \ref{rem:valeurc2}), $W^j$, $j=1,2$ can be realized as follows on $[s_{m,\kbis},s_{m,\kbis+1}]$;
$$\forall t\in[s_{m,\kbis},s_{m,\kbis+1}], \quad W^j_t=(U_j^{m,\kbis}+V_{m,\kbis})\frac{\int_{s_{m,\kbis}}^{t} {g_{_S}}(s)ds}{\|g_{_S}\|_{[s_{m,\kbis},s_{m,\kbis+1}],2}}+\tilde{W}^{m}_{t-s_{m,2}},$$
where $(\tilde{W}_t^m)_{t\ge0}$ is a standard Brownian motion, $(V_{m,\kbis})_{\kbis\ge1}$ is a sequence of $i.i.d.$ normally distributed random variables, 
and 
$$\forall \kbis\in\{2,\ldots,\ell-1\},\; \forall \omega\in {\cal A}_{m,\ell},\; |U_{j}^{m,\kbis}(\omega)|\le \frac{1}{2}\max\{ 2^{2-\alpha v},-2\log (\frac{2^{-\alpha \kbis}}{8})\}+2^{-\alpha v}\le C \log (2^{\alpha v})$$ 
 where $C$ does not only depend on $\alpha$. Furthermore, $(\tilde{W}_t^m)_{t\ge0}$  and $(V_{m,\kbis})_{\kbis\ge1}$
are independent of  $(U_1^{m,\kbis},U_2^{m,\kbis})_{\kbis}$ and $g_{_S}$. In particular, $(\tilde{W}_t^m)_{t\ge0}$  and $(V_{m,\kbis})_{\kbis\ge1}$ are independent of ${\cal A}_{m,\ell}$.
Set $s_{m,\kbis}^u=s_{m,\kbis}\vee (s_{m,\ell}-u)$. The above properties imply that for every  { $u\in[0, s_{m,\ell}-s_{m,2}]$, } \
\bqn\label{ddiooei}
\begin{split}
\ES[|W^j_{s_{m,\ell}}-W^j_{s_{m,\ell}-u}|^p|{\cal A}_{m,\ell}]&\le C_p\ES\left[\left(\sum_{\kbis=2}^{\ell-1} \log (2^{\alpha \kbis})\frac{\int_{s_{m,\kbis}^u}^{s_{m,\kbis+1}^u} |g_{_S}(s)|ds}{\|g_{_S}\|_{[s_{m,\kbis},s_{m,\kbis+1}],2}}\right)^p|{\cal A}_{m,\ell}\right]\\
&+C_p\ES\left[\left|\sum_{\kbis=2}^\ell V_{m,\kbis}\frac{\int_{s_{m,\kbis}^u}^{s_{m,\kbis+1}^u} g_{_S}(s)ds}{\|g_{_S}\|_{[s_{m,\kbis},s_{m,\kbis+1}],2}}\right|^p|{\cal A}_{m,\ell}\right]\\
&+C_p\ES[ |\tilde{W}^m_{s_{m,\ell}-s_{m,2}}-\tilde{W}^m_{s_{m,\ell}-s_{m,2}-u}|^p].
\end{split}
\eqn  
We focus successively on each term of the right-hand side of the above inequality. First,
\bqn\label{eq:bro1}
\ES[ |\tilde{W}^m_{s_{m,\ell}-s_{m,2}}-\tilde{W}^m_{s_{m,\ell}-s_{m,2}-u}|^p]=u^{\frac{p}{2}} \ES[|U|^p]
\eqn
where $U$ stands for a normally distributed random variable.

\smallskip
\noindent For the first right-hand member term of \eqref{ddiooei}, we deduce from the Cauchy-Schwarz inequality that
\bqne
\sum_{\kbis=2}^{\ell-1} \log (2^{\alpha \kbis})\frac{\int_{s_{m,\kbis}^u}^{s_{m,\kbis+1}^u} |g_{_S}(s)|ds}{\|g_{_S}\|_{[s_{m,\kbis},s_{m,\kbis+1}],2}}
\le \left(\sum_{\kbis=2}^{\ell-1} \log (2^{\alpha \kbis})^2\right)^{\frac{1}{2}}\left(\sum_{\kbis=2}^{\ell-1}\left(\frac{\int_{s_{m,\kbis}^u}^{s_{m,\kbis+1}^u} |g_{_S}(s)|ds}{\|g_{_S}\|_{[s_{m,\kbis},s_{m,\kbis+1}],2}}\right)^2\right)^{\frac{1}{2}}
\eqne
Using that $\kbis\mapsto\log (2^{\alpha \kbis})$ is non-decreasing and that 
\bqn\label{eq:codni}
\left(\int_{s_{m,\kbis}^u}^{s_{m,\kbis+1}^u} |g_{_S}(s)|ds\right)^2\le (s_{m,\kbis+1}^u-s_{m,\kbis}^u)\left(\|g_{_S}\|_{[s_{m,\kbis},s_{m,\kbis+1}],2}\right)^2,
\eqn
we deduce that 
\bqne
\sum_{\kbis=2}^{\ell-1} (\log (2^{\alpha \kbis}))\frac{\int_{s_{m,\kbis}^u}^{s_{m,\kbis+1}^u} |g_{_S}(s)|ds}{\|g_{_S}\|_{[s_{m,\kbis},s_{m,\kbis+1}],2}}
\le \sqrt{\ell}\log (2^{\alpha \ell})u^{\frac{1}{2}}.
\eqne
Using that for all $\ell \ge 3$, $s_{m,\ell}-s_{m,2}= c_2 2^{\ell-1}$ with $c_2\ge 1$, one deduces that  for every positive $p$ and $\varepsilon$, there exists $C_{p,\varepsilon}$ such that for every $\ell\ge 3$,  $\sqrt{\ell}\log (2^{\alpha \ell})\le C_{p,\varepsilon} (s_{m,\ell}-s_{m,2})^{\frac{\varepsilon}{p}}$ (we recall that $\alpha$ is a fixed number of $(0,1/2)$).  As a consequence,
the first right-hand member term of \eqref{ddiooei} satisfies for every $u\in[0, s_{m,\ell}-s_{m,2}]$
\bqn\label{eq:bro2}
\ES\left[\left(\sum_{\kbis=2}^{\ell-1} \log (2^{\alpha \kbis})\frac{\int_{s_{m,\kbis}^u}^{s_{m,\kbis+1}^u} |g_{_S}(s)|ds}{\|g_{_S}\|_{[s_{m,\kbis},s_{m,\kbis+1}],2}}\right)^p|{\cal A}_{m,\ell}\right]\le C_{p,\varepsilon}(s_{m,\ell}-s_{m,2})^{{\varepsilon}} u^{\frac{p}{2}}
\eqn
where $C_{p,\varepsilon}$ is the constant defined above.
Finally, for the second right-hand member term of \eqref{ddiooei}, let us define  $(X_{m,\kbis}^u)_{\kbis=2}^{\ell-1}$ by
$$\forall \kbis\in\{2,\ldots,\ell-1\},\qquad X_{m,\kbis}^u=V_{m,\kbis}\frac{\int_{s_{m,\kbis}^u}^{s_{m,\kbis+1}^u} g(s)ds}{\|g\|_{[s_{m,\kbis},s_{m,\kbis+1}],2}}.$$
Since $(V_{m,\kbis})_{\kbis\ge1}$ is centered and independent of $g_{_S}$ and ${\cal A}_{m,l}$,  it follows that $(X_{m,\kbis}^u)_{\kbis=2}^{\ell-1}$ is a sequence of  martingale increments under $\PE(.|{\cal A}_{m,\ell})$. By the Doob inequality and \eqref{eq:codni}, we deduce that,
\begin{align}
\ES\left[\left|\sum_{\kbis=2}^{\ell-1} V_{m,\kbis}\frac{\int_{s_{m,\kbis}^u}^{s_{m,\kbis+1}^u} g_{_S}(s)ds}{\|g_{_S}\|_{[s_{m,\kbis},s_{m,\kbis+1}],2}}\right|^p|{\cal A}_{m,\ell}\right]
&\le C_p\ES\left[\left(\sum_{\kbis=2}^{\ell-1} \left(\frac{\int_{s_{m,\kbis}^u}^{s_{m,\kbis+1}^u} g_{_S}(s)ds}{\|g_{_S}\|_{[s_{m,\kbis},s_{m,\kbis+1}],2}}\right)^2\right)^{\frac{p}{2}}|{\cal A}_{m,\ell} \right]\nonumber\\
&\le C_p \left(\sum_{\kbis=2}^{\ell-1} (s_{m,\kbis+1}^u-s_{m,\kbis}^u)\right)^{\frac{p}{2}}\le C_p u^{\frac{p}{2}}.\label{eq:bro3}
\end{align}

\smallskip
\noindent By \eqref{eq:bro1}, \eqref{eq:bro2} and \eqref{eq:bro3}, we obtain that there exists $C_{p,\varepsilon}\in\ER_+^*$ such that for all $m\in\EN$, $\ell\ge2$,
$j\in\{1,2\}$ and $u\in[0,s_{m,\ell}-s_{m,2}]$,
\bqn\label{eq:bro4}
\ES[|W^j_{s_{m,\ell}}-W^j_{s_{m,\ell}-u}|^p|{\cal A}_{m,\ell}]\le C_{p,\varepsilon} (s_{m,\ell}-s_{m,2})^{\varepsilon} u^{\frac{p}{2}}.
\eqn
The results follows by plugging this inequality into \eqref{eq:etape2simp}.

\smallskip
\noindent \underline{$i=2$:} 
Here, we consider the interval where Step 2 fails. With the previous notations, $\tau_{m}^2=s_{m,\ell}$ and $\tau_{m}^3=s_{m,\ell+1}$ on ${\cal A}_{m,\ell}$ when $\ell\ge 2$. By a similar strategy as in the case $i=1$ (see in particular \eqref{eq:etape2simp} and \eqref{eq:bro4}), 
it is enough to show that there exists $C_{p,\varepsilon}$ such that
for all $m\in\EN$, $\ell\ge 0$,
$j\in\{1,2\}$ and $u\in[0,s_{m,\ell+1}-s_{m,\ell}]$,
\bqn\label{eq:bro5}
\ES[|W^j_{s_{m,\ell+1}}-W^j_{s_{m,\ell+1}-u}|^p|{\cal A}_{m,\ell}]\le C_{p,\varepsilon} (s_{m,\ell+1}-s_{m,\ell})^{\varepsilon} u^{\frac{p}{2}}.
\eqn
When $\ell=0,1$, $\tau_{m}^3-\tau_{m}^2=0$ so that the property is obvious. 
Let us consider the set ${\cal B}_{m,\ell-1}$ defined by \eqref{eq:bkl}. From the very definition, ${\cal A}_{m,\ell} \subset {\cal B}_{m,\ell-1}$.  

By Hölder inequality applied with some $\tilde{p}>1$ and $\tilde{q}>1$ such that $1/\tilde{p}+1/\tilde{q}=1$, we have
$$
\ES[|W^j_{s_{m,\ell+1}}-W^j_{s_{m,\ell+1}-u}|^p|{\cal A}_{m,\ell}]\le \ES[|W^j_{s_{m,\ell+1}}-W^j_{s_{m,\ell+1}-u}|^{p\tilde{p}}|{\cal B}_{m,\ell-1}]^{\frac{1}{\tilde{p}}}\PE({\cal A}_{m,\ell}|{\cal B}_{m,\ell-1})^{ \frac{1}{\tilde{q}}-1}.$$
On the one hand, we deduce from the independence of the increments of the Brownian motion  that
$$\ES[|W^j_{s_{m,\ell+1}}-W^j_{s_{m,\ell+1}-u}|^{p\tilde{p}}|{\cal B}_{m,\ell-1}]^{\frac{1}{\tilde{p}}}=u^{\frac{p}{2}} \ES[|U|^{p\tilde{p}}]^{\frac{1}{\tilde{p}}}$$
where $U$ stands for a normally distributed random variable.  
%
On the other hand, by \eqref{eq:lbbk1},
$$\PE({\cal A}_{m,\ell}|{\cal B}_{m,\ell-1})=\PE({\cal B}_{m,\ell}^c|{\cal B}_{m,\ell-1})\ge 2^{-\alpha \ell -1}.$$
Then, for each $\varepsilon\in(0,1)$, Inequality \eqref{eq:bro5} follows by setting $\tilde{q}=(1-\varepsilon)^{-1}$ (so that $1-\frac{1}{\tilde{q}}=\varepsilon$ ). 

\smallskip
\noindent \underline{$i=3$:} This corresponds to Step 3. The key point here is that the increments of the Brownian motion after $\tau_{m}^3$ are independent of the previous coupling attempt so that, denoting by $\Delta_3(m,\ell)$, the length of Step 3 under ${\cal A}_{m,\ell}$, we have 
\bqn\label{eq:bro6}
\ES[|W^j_{\tau_m}-W^j_{\tau_m-u}|^p|{\cal A}_{m,\ell}]=
\ES[|W^j_{\tau_m^3+\Delta_3(m,\ell)}-W^j_{\tau_m^3+\Delta_3(m,\ell)-u}|^p]= C_{p} u^{\frac{p}{2}}.
\eqn
The result then follows similarly to the case $i=1$ (see \eqref{eq:etape2simp}).

\smallskip

\noindent $(ii)$  This result can be easily derived from the controls established previously. More precisely, cases $i=0,1,2,3$ can be viewed as particular cases of  \eqref{eq:bro0}, \eqref{eq:bro4}, \eqref{eq:bro5} and \eqref{eq:bro6}. 

%

%
%
%
%
%

%

%
\end{proof}
In the next lemma, we adopt the convention $\sum_{\emptyset}=1$. Also, let us recall that by \eqref{durationstep3}, $\Delta \tau_k\ge c_3 a_k\ge a_k$ since 
$c_3\ge 2 c_2\ge 1$ (and that $c_2\ge1$  by Remark \ref{rem:valeurc2}).
\begin{lemme}\label{lemme:lambdamk}
Let $\Psi$ satisfy  $\int\Psi(x_1)\bar\mu_1(dx_1)+\int\Psi(x_2)\bar\mu_2(dx_2)<+\infty$ and assume that $\tau_0$ is defined in terms of this function as in \eqref{def:tau0}.  Assume that there exists ${\delta_1}>0$ such that for every $m\ge1$ and $K>0$, $\PE({\cal E}_m|{\cal E}_{m-1})\ge \delta_1\in(0,1)$.
Then, for $j=1,2$ and  {for every $\rho\in (0,1)$} and $\varepsilon\in(0,1-H)$, there exists $C$ such that for every $k\ge 1$, $m\in\{0,\ldots, k-1\}$ and $K>0$,
\begin{equation}\label{eq:pastbf}
\ES[\sup_{(s,t), \tau_{k-1}\le s<t\le s+1}\frac{1}{t-s}|\Lambda_{m,k}|(s,t,W^j) |{\cal E}_k]\le C\displaystyle{\frac{\left(\sum_{\ell=m+1}^{k-1} a_\ell\right)^{H-1+\varepsilon}}{\delta_1^{\rho(k-m)}}}.
\end{equation}

\smallskip
\noindent As a consequence, if $a_k= \varsigma^{k}$ where $\varsigma\in(1,+\infty)$, there exists $\Delta\in(0,1)$ and $C>0$  such that
for every integers $m$ and $k$ with $m\le k$
\begin{equation}\label{eq:pastbf2}
\ES[\sup_{(s,t), \tau_{k-1}\le s<t\le s+1}\frac{1}{t-s}|\Lambda_{m,k}|(s,t,W^j)|{\cal E}_k]\le C\Delta^{k-m}.
\end{equation}
\end{lemme}
\begin{Remarque} $\rhd$ The assumption on the moments of $\muchap$ is only necessary for the case $m=0$ which corresponds to the interval $[-\infty,\tau_0]$. Due to the memory, $\tau_0$ is not independent of the past of the Brownian Motion before $\tau_0$. But the assumption on $\muchap$ leads to a control of the queue of $\tau_0$
which is sufficient to overcome the non-independence property.  
\smallskip
 
\noindent $\rhd$ The fact that the  quality of the estimate strongly decreases with $m-k$ may appear surprising. The main problem is that we do not have 
a sharp idea of the distribution of { ${\cal L}(W_t-W_{\tau_{m-1}},\tau_{m-1}\le t\le\tau_m)$} conditionally to the event $\{\Delta \tau_l<+\infty, m\le l<k\}$ and thus, we compensate this failure by some Hölder-type inequalities. 
 
$\rhd$ The second statement says that if one waits sufficiently between each trial, the influence of the past decreases geometrically with $m$. Note that this waiting time increases geometrically. This may be a problem for the sequel and the fact that $\varsigma$ can be chosen arbitrarily close to $1$ will be of first importance.
\end{Remarque}
\begin{proof}
First, note that if \eqref{eq:pastbf} is true, \eqref{eq:pastbf2} easily follows: let $\varsigma>1$ and let $\gamma_1\in(0,+\infty)$ be  such that 
$\varsigma=\delta_1^{-\gamma_1}$.  It is now sufficient to remark that for every $m\in\{1,\ldots, k-2\}$,
$$(\sum_{\ell=m+1}^{k-1} a_\ell)^{H-1+\varepsilon}\delta_1^{-\rho(k-m)}\le \delta_1^{\left(\gamma_1(1-H-\varepsilon)-\rho\right)(k-m)}$$
and to choose for instance $\varepsilon=(1-H)/2$ and $\rho=\frac{\gamma_1(1-H)}{4}\wedge\frac{1}{2}$ so that 
$$\gamma_1(1-H-\varepsilon)-\rho\ge \frac{\gamma_1(1-H)}{4}\wedge \frac{1}{2}>0.$$
Let us now prove \eqref{eq:pastbf}. We consider  three cases:

\smallskip
\noindent\textbf{Case 1}: $k\ge3$ and $m\in\{1,\ldots,k-2\}$.
 $m\in \EN$. In harmony with Lemma \ref{lemma:lembrown}, we decompose 
$[\tau_{m-1},\tau_m]$ in four intervals $[\tau_{m}^i, \tau_{m}^{i+1}]$,  $i\in\{0,1,2,3\}$, and also cut $\Lambda_{m,k}$ (which does not depend on $k$ in this case)
in four parts denoted by $\phi_{m,i}$:
\bqne
\forall i\in\{0,1,2,3\},\quad\phi_{m,i}(s,t,W^j)=
\int_{\tau_{m}^i}^{\tau_{m}^{i+1}} (t-r)^{H-\frac{1}{2}}-(s-r)^{H-\frac{1}{2}} d W^j_r.
\eqne 
By Lemma \ref{lemme:IPP}, 

$$ \frac{1}{t-s}|\phi_{m,i}(s,t,W^j)|\le (t-\tau_{m}^i)^{H-\frac{3}{2}}|W_{\tau_{m}^i}-W_{\tau_{m}^{i+1}}|+\frac{1}{2}\int_{\tau_{m}^i}^{\tau_{m}^{i+1}}
 (s-r)^{H-\frac{5}{2}}|W_r-W_{\tau_{m}^{i+1}}|dr.$$
If $\tau_{k-1}\le s <t$, since $\Delta \tau_l\ge a_l$ for all $l$, we get
$$t-\tau_{m}^i\ge \max (\sum_{\ell=m+1}^{k-1} a_\ell, \tau_{m}^{i+1}-\tau_{m}^i)\;\textnormal{and}\;   \forall r\in[\tau_{m}^i,\tau_{m}^{i+1}],\; s-r\ge \max (\sum_{\ell=m+1}^{k-1} a_\ell, 1+\tau_{m}^{i+1}-r).$$
Thus, for every $\varepsilon\in(0,1-H)$,
$$\sup_{(s,t),  { \tau_{k-1}}\le s<t\le s+1} \frac{1}{t-s}|\phi_{m,i}(s,t,W^j)|\le C (\sum_{\ell=m+1}^{k-1} a_\ell)^{H-1+\varepsilon} \Xi_{m,\varepsilon,i}$$
where 
$$\Xi_{m,\varepsilon,i}=(\tau_{m}^{i+1}-\tau_{m}^i)^{-\frac{1}{2}-\varepsilon}|W^j_{\tau_{m}^i}-W^j_{\tau_{m}^{i+1}}|+\frac{1}{2}\int_{\tau_{m}^i}^{\tau_{m}^{i+1}} (1+\tau_{m}^{i+1}-r)^{-\left(\frac{3}{2}+\varepsilon\right)}|W^j_r-W^j_{\tau_{m}^{i+1}}|dr.$$
The interesting point is that $\Xi_{m,\varepsilon,i}$ does not depend on $k$. Furthermore, by Lemma \ref{lemma:lembrown}, for every $p>1$, $\varepsilon\in(0,1-H)$
and $i\in\{0,1,2,3\}$,
$$\ES[\left|\Xi_{m,\varepsilon,i}\right|^p|{\cal E}_m]^{\frac{1}{p}}\le C$$
where $C$ does only depend on $p$, $\varepsilon$ and $H$.
Set $\Xi_{m,\varepsilon}=\sum_{i=0}^3 \Xi_{m,\varepsilon,i}$. Summing up the previous controls (on $i$), we deduce from Hölder inequality that for every $p>1$ and $q>1$ such that $\frac{1}{p}+\frac{1}{q}=1$,
\bqn\label{eq:rettopast}
\begin{split}
0\le \ES[\sup_{(s,t), { \tau_{k-1}}\le s<t\le s+1} \frac{1}{t-s}|\Lambda_{m,k}(s,t,W^j)||{\cal E}_{k}]&\le \ES[|\Xi_{m,\varepsilon}|^p|{\cal E}_{m}]^{\frac{1}{p}}(\sum_{\ell=m+1}^{k-1} a_\ell)^{H-1+\varepsilon}\PE({\cal E}_k|{\cal E}_{m})^{\frac{1}{q}-1}\\
&\le C \left(\frac{1}{\delta_1}\right)^{\left(1-\frac{1}{q}\right)(k-m)}(\sum_{\ell=m+1}^{k-1} a_\ell)^{H-1+\varepsilon}. 
\end{split}
\eqn
The result follows in this case by noticing that for every $\rho\in(0,1)$, there exists $q\in(1,+\infty)$ such that $\rho=1-1/q$.
\smallskip

\noindent\textbf{Case 2}: $k\ge2$ and $m=k-1$. It corresponds to the integral on the interval $[\tau_{k-2},\tau_{k-1}-1]$. The proof is almost identical using the controls
$$t-\tau_{m}^i\ge 1+\tau_{m}^{i+1}\wedge (\tau_{k-1}-1)-\tau_{m}^i\quad\textnormal{and}\quad   \forall r\in[\tau_{m}^i,\tau_{m}^{i+1}\wedge(\tau_{k-1}-1)],  \, s-r\ge  1+\tau_{m}^{i+1}\wedge (\tau_{k-1}-1)-r.$$
We do not detail it.

\smallskip
\noindent\textbf{Case 3}: $k\ge1$ and $m=0$. It corresponds to the integral on the interval $(-\infty,\tau_0]$ if $k\ge2$ and $(-\infty,\tau_0-1]$ if $k=1$. For the sake of simplicity, we only consider the case $k\ge2$. Note that on this interval $W^1=W^2$. We then write $W$ only. By Lemma \ref{lemme:IPP} and the fact that $\lim_{M\rightarrow+\infty} M^{-1/2-\varepsilon} W_{-M}=0$, we have

$$ \frac{1}{t-s}|\Lambda_{0,k}(s,t,W)|\le\frac{1}{2}\int_{-\infty}^{\tau_0}
 (s-r)^{H-\frac{5}{2}}|W_r-W_{\tau_0}|dr$$
 so that 
 $$\sup_{\tau_{k-1}\le s<t\le \lfloor s\rfloor+1}\frac{1}{t-s}|\Lambda_{0,k}(s,t,W)|\le\frac{1}{2}\left(\sum_{m=1}^{k-1} a_k\right)^{H-1+\varepsilon}\int_{-\infty}^{\tau_0}
 (1+\tau_0- r)^{-\frac{3}{2}-\varepsilon}|W_r-W_{\tau_0}|dr$$
where $\varepsilon\in(0,1-H)$. Let $p\ge1$. As remarked previously, one has no information about the joint law of $\tau_0$ and $W^j$. We compensate this failure by a rough argument. Using Cauchy-Schwarz inequality,
\begin{align*}
\ES_{\muchap}\Big[\Big(\int_{-\infty}^{\tau_0}
 (1+\tau_0- r)^{-\frac{1}{2}-\varepsilon}&|W_r-W_{\tau_0}|dr\Big)^p\Big]\\
&\le\sum_{u=1}^{+\infty}\ES\left[\left(\int_{-\infty}^{u}
 (1+u-r)^{-\frac{3}{2}-\varepsilon}|W_r-W_{u}|dr\right)^{2p}\right]^{\frac{1}{2}}\PE_{\muchap}({\tau_0= u})^{\frac{1}{2}}.
\end{align*}
Thanks to the stationarity of the increments of the Brownian motion, we deduce from a change of variable that 
$$\ES\left[\left(\int_{-\infty}^{u}
 (1+u-r)^{-\frac{3}{2}-\varepsilon}|W_r-W_{u}|dr\right)^{2p}\right]\le \ES\left[\left(\int_{0}^{+\infty}
 (1+r)^{-\frac{3}{2}-\varepsilon}|W_r|dr\right)^{2p}\right]=:C_p.$$
Using that for every $p\ge1$ and $\varepsilon>0$, 
$$\ES[\sup_{r\ge1} \left(\frac{|W_r|}{r^{\frac{1+\varepsilon}{2}}}\right)^{2p}]<+\infty,$$
we deduce that $C_p$ is finite. It remains to show that $\sum_{u=1}^{+\infty}\PE_{\muchap}(\tau_0= u)^{\frac{1}{2}}<+\infty$.  This property has already been proved in 
\eqref{def:tau0}.
\end{proof}
\begin{prop}\label{prop:contdubruit} Assume that Step $1$ and $2$ are carried out as described in Section 3. Let $\Psi$ satisfy  $\int\Psi(x_1)\bar\mu_1(dx_1)+\int\Psi(x_2)\bar\mu_2(dx_2)<+\infty$ and assume that $\tau_0$ is defined in terms of this function as in  \eqref{def:tau0}.  Assume there exists ${\delta_1}>0$ such that for every $m\ge1$, $\PE({\cal E}_m|{\cal E}_{m-1})\ge \delta_1\in(0,1)$ and that for every $k\ge1$, $a_k$ defined in \eqref{durationstep3} satisfies $a_k=\varsigma^{k}$ with $\varsigma>1$. Then, $\Hpdeux$ holds for every $\theta\in(1/2,H)$.
\end{prop}
\begin{proof} First, thanks to a change of variable and to the decomposition introduced at the beginning of the current subsection \ref{subsec:hpdeux}, we have
\bqn\label{djksls}
\begin{split}
\sum_{u=1}^{\Delta \tau_k} \rho^{\Delta \tau_k -u}\|B\|_\theta^{\tau_{k-1}+u-1,\tau_{k-1}+u}&\le \sum_{m=1}^k
\sum_{u=\tau_{k-1}+1}^{\tau_k} \rho^{\tau_k -u} \|\Lambda_{m,k}^j\|_\theta^{u-1,u}\\
&+\sum_{m=1}^3
\sum_{u=\tau_{k-1}+1}^{\tau_k} \rho^{\tau_k -u} \|\Gamma_{m}^j\|_\theta^{u-1,u}
\end{split}
\eqn
where we used the notations introduced in \eqref{eq:notabus}. 
Thus, the idea is to bound each term of the right-hand side. 
First, for every $m\in\{0,\ldots,k-1\}$, for every $u\in\{\tau_{k-1},\ldots,\tau_k\}$,
$$\|\Lambda_{m,k}^j\|_\theta^{u,u+1}\le \sup_{\tau_{k-1}\le s<t\le \lfloor s\rfloor +1}\frac{1}{t-s}|\Lambda_{m,k}(s,t,W^j)|$$
Since the right-hand member does not depend on $u$, we deduce that for every $m\in\{0,\ldots,k-1\}$
\bqn\label{fieofjj}
\sum_{u=\tau_{k-1}+1}^{\tau_k} \rho^{\tau_k -u}  \|\Lambda_{m,k}^j\|_\theta^{u-1,u}\le \sup_{\tau_{k-1}\le s<t\le \lfloor s\rfloor +1}\frac{1}{t-s}|\Lambda_{m,k}(s,t,W^j)|\sum_{w=0}^{+\infty} \rho^w.
\eqn
Thus, by Lemma \ref{lemme:lambdamk}, it follows that for every $m\in\{0,\ldots,k-1\}$
\bqne
\ES[\sum_{u=\tau_{k-1}+1}^{\tau_k} \rho^{\tau_k -u}  \|\Lambda_{m,k}^j\|_\theta^{u-1,u}|{\cal E}_k]\le \frac{C}{1-\rho}\Delta^{k-m}
\eqne
where $\Delta\in(0,1)$. As a consequence, 
\bqne
\ES[\sum_{m=1}^{k-1}
\sum_{u=\tau_{k-1}+1}^{\tau_k} \rho^{\tau_k -u} \|\Lambda_{m,k}^j\|_\theta^{u-1,u}\\|{\cal E}_k]\le \frac{C}{(1-\rho)(1-\Delta)}.
\eqne
%
Keeping in mind inequality \eqref{djksls}, it remains to bound, independently of $k$ and $K$, the terms involving $\Lambda_{k,k}$ and $\Gamma_m$, $m=1,2,3$. The strategy is  different since these terms depend on the path between $\tau_{k-1}-1$ and $\tau_{k}$.  Let us begin by $\Lambda_{k,k}$. By a change of variable,
$$\sum_{u=\tau_{k-1}+1}^{\tau_k} \rho^{\tau_k -u} \|\Lambda_{k,k}^j\|_\theta^{u-1,u}=\sum_{v=0}^{\Delta \tau_k-1} \rho^{\Delta \tau_k-1 -v}
Z_{k,v}\quad\textnormal{where}\quad Z_{k,v}=\|\Lambda_{k,k}^j\|_\theta^{\tau_{k-1}+v,\tau_{k-1}+v+1}.$$
Note that $Z_{k,0}=0$. On the event ${\cal A}_{k,\ell}$, one knows that $\Delta \tau_k$ is deterministic. Denote it by $\Delta(k,\ell)$. 
Decomposing the event ${\cal E}_k$, it follows that
$$\ES[\sum_{u=\tau_{k-1}+1}^{\tau_k} \rho^{\tau_k -u} \|\Lambda_{k,k}^j\|_\theta^{u-1,u}|{\cal E}_{k}]=
\sum_{\ell\ge0}\sum_{v=1}^{\Delta(k,\ell)}\rho^{\Delta(k,\ell)-1 -v}\ES[ Z_{k,v}|{\cal A}_{k,\ell}]\PE({\cal A}_{k,\ell}|{\cal E}_k).$$
Using that ${\cal A}_{k,\ell}\subset {\cal E}_{k}\subset {\cal E}_{k-1}$ and Cauchy-Schwarz inequality, we remark that
\bqn\label{argumentrough}
\ES[ Z_{k,v}|{\cal A}_{k,\ell}] \PE({\cal A}_{k,\ell}|{\cal E}_k)\le {\ES[ Z_{k,v}^2|{\cal E}_{k}]^{\frac{1}{2}}}{\PE({\cal A}_{k,\ell}|{\cal E}_k)^{\frac{1}{2}}}
\le \frac{\ES[ Z_{k,v}^4|{\cal E}_{k-1}]^{\frac{1}{4}}}{\PE({\cal E}_{k}|{\cal E}_{k-1})^{\frac{1}{4}}}\PE({\cal A}_{k,\ell}|{\cal E}_k)^{\frac{1}{2}}.
\eqn
But $\PE({\cal E}_k|{\cal E}_{k-1})\ge \delta_1>0$ and using Lemma \ref{lemme:step2.2}, we have for every $\ell\ge 2$,
\bqn\label{pakl}
\PE({\cal A}_{k,\ell}|{\cal E}_k)=\frac{\PE({\cal A}_{k,\ell}|{\cal E}_{k-1})}{\PE({\cal E}_{k}|{\cal E}_{k-1})}\le \delta_1^{-1}\PE({\cal B}_{k,\ell-1}|{\cal E}_{k-1})\PE({\cal B}_{k,\ell}^c|{\cal B}_{k,\ell-1})\le \delta_1^{-1}2^{-\alpha \ell}
\eqn
so that
\bqn\label{dsopoids}
\ES[\sum_{u=\tau_{k-1}+1}^{\tau_k} \rho^{\tau_k -u} \|\Lambda_{k,k}^j\|_\theta^{u-1,u}|{\cal E}_{k}]\le C_{\delta_1,\rho}\sup_{v\in \EN} \ES[ Z_{k,v}^4|{\cal E}_{k-1}]^{\frac{1}{4}}
\eqn
where $C_{\delta_1,\rho}$ is a finite constant depending only on $\delta_1$ and $\rho$. Set $\varepsilon_\theta=(H-\theta)/2$. Using that $(W_{u+\tau_{k-1}-1}-W_{\tau_{k-1}-1})_{u\ge0}$ is independent of ${\cal F}_{\tau_{k-1}-1}$, we obtain for every $v\in \EN$, for every $k\in \EN$,
$$\ES[ Z_{k,v}^4|{\cal E}_{k-1}]\le \ES[(F_v(\bar{W}_u, 0\le u\le v))^4]$$
where $\bar{W}$ is a standard Brownian Motion and $F_v: {\cal C}^{\frac{1}{2}-\varepsilon_\theta}([0,v],\ER^d)\rightarrow\ER$ defined by
$$ F_v(w_u,0\le u\le v)=\sup_{v+1\le s<t\le v+2]}\frac{1}{t-s}\left|\int_{0}^{v }(t+1-r)^{H-\frac{1}{2}}-(s+1-r)^{H-\frac{1}{2}} dw_r\right|.$$
By Lemma \ref{lemme:IPP}, for every $v\ge 1$,
$$F_v(\bar{W}_u,0\le u\le v)\le v^{H-\frac{3}{2}}|\bar{W}_{v}|+\frac{1}{2}\int_0^{v} (v+1-r)^{H-\frac{5}{2}}|\bar{W}_r-\bar{W}_{v}|dr.$$
Denote by $(\tilde{W}_u)_{u\in[0,1]}$ the rescaled Brownian motion defined by  $\tilde{W}_u=\sqrt{v}\bar{W}_{u v}$. By a change of variable,  for every 
$v\ge 1$, 
$$F_v(\bar{W}_u,0\le u\le v)\le |\tilde{W}_1|+\frac{1}{2}\int_0^{1} \sqrt{v} (v+1-u v)^{H-\frac{5}{2}}|\tilde{W}_u-\tilde{W}_1| du.$$
%
%
Checking that 
$$\forall v\ge 1,\quad  \int_0^{1} \sqrt{v} (v+1-u v)^{H-\frac{5}{2}} du \le (\frac{3}{2}-H)^{-1},$$
one deduces that 

$$\sup_{v\ge 1}\ES[ Z_{k,v}^4|{\cal E}_{k-1}]\le \ES\left[\left(|\tilde{W}_1|+\frac{2}{\frac{3}{2}-H} \sup_{u\in[0,1]}|\tilde{W}_u-\tilde{W}_1|)^4\right)\right].$$
By plugging this inequality into \eqref{dsopoids}, this concludes the study of $\Lambda_{k,k}$.

\smallskip
\noindent As concerns $\Gamma_{i}$, $i=1,2,3$, one deduces with a similar strategy as before that it is enough to show that for $i=1,2,3$,
\bqn\label{eq:gi1}
\ES[|G_i(\bar{W}_u,0\le u\le 2)|^4]\le C,
\eqn
where, setting $h=t-s$, 
\begin{align*}
&G_1(w)=\sup_{1\le s<t\le 2}\frac{1}{(t-s)^\theta}\left|\int_{0}^{s-h} (t-r)^{H-\frac{1}{2}}-(s-r)^{H-\frac{1}{2}} dw_r\right|,\\
&G_2(w)=\sup_{1\le s<t\le 2}\frac{1}{(t-s)^{\theta}}\left|\int_{s-h}^{s} (s-r)^{H-\frac{1}{2}} dw_r\right|\quad\textnormal{and}\;\\
& G_3(w)=\sup_{1\le s<t\le 2}\frac{1}{(t-s)^{\theta}}\left|\int_{s-h}^{t} (t-r)^{H-\frac{1}{2}} dw_r\right|.
\end{align*}
By Lemma \ref{lemme:IPP}, we have 
\begin{align*}
\int_{0}^{s-h} (t-r)^{H-\frac{1}{2}}-(s-r)^{H-\frac{1}{2}} dw_r&\le |w_{s-h}-w_0|(t-s)\\
&+(t-s)^{\theta}\int_{0}^{s-h} (s-h-r)^{H-\theta-\frac{3}{2}}|w_r-w_{s-h}|dr
\end{align*}
where in the second line, we used that for every $r\in[0,{s-h}]$, $s-r\ge s-h-r$ and $s-r\ge h$.  We deduce that
\bqn\label{eq:contg1}
G_1(w)\le  C(1+\|w\|_{\frac{1}{2}-\varepsilon_\theta}^{0,2})\quad \textnormal{with}\quad\varepsilon_\theta=\frac{H-\theta}{2}.
\eqn
and \eqref{eq:gi1} follows for $i=1$. By an integration by parts, one also checks that
\begin{align*}
\left|\int_{s-h}^{s} (s-r)^{H-\frac{1}{2}} dw_r\right|&\le \|w\|_{\frac{1}{2}- \varepsilon_\theta}^{0,2}\left(h^{H-\varepsilon_\theta}+ (H-\frac{1}{2})\int_{s-h}^s (s-r)^{H-\frac{3}{2}+\frac{1}{2}-\varepsilon_\theta} dr\right)\\
&\le Ch^{H-\varepsilon_\theta}\|w\|_{\frac{1}{2}-\varepsilon_\theta}^{0,2}.
\end{align*} The previous control also holds for $G_3(w)$ and since $H-\varepsilon_\theta\ge\theta$, it follows that 
\eqref{eq:gi1} is also true for $i=2,3$.

\end{proof}
%
%


%
\subsection{Condition \eqref{hairerassumpcond}}\label{subsec:hairerassum}
\begin{prop}\label{lemme:step3}  Let $\alpha\in(0,1/2)$  and assume that for every $k\ge1$ and $\ell\ge0$, $\Delta_3(\ell,k)=c_3 a_k 2^{\beta \ell}$ with  $\beta>(1-2\alpha)^{-1}$, $a_k=\varsigma^k$ with an arbitrary $\varsigma>1$.  Then, for each $K>0$, there is a choice of $c_3$ such that, for  every   $k\ge 0$, Condition \eqref{hairerassumpcond}
is $a.s.$ satisfied at time $\tau_k$   on the event $\{ \tau_k<+\infty \}$. In other words, for every $k\ge0$, $\PE(\Omega_{\alpha,\tau_k}^1|\tau_k<+\infty)=1$.
\end{prop}
\begin{proof} For every $k\ge0$, set
$$ u_k:=\sup_{T\ge0} \left( \int_0^{+\infty} (1+t)^{2\alpha}|({\cal R}_T g_w^{\tau_{k}})(t)|^2  dt\right)^{\frac{1}{2}}.$$
Since $g_w$ is  null on $(-\infty,\tau_0]$, $u_0=0$. Following carefully the proof of Lemma 5.15 in \cite{hairer} (see (5.8)  therein for notation), we check  that there exists $C>0$ such that for every  $k\ge1$,
\begin{equation*}
\begin{split}  u_{k}-u_{k-1}\le &  \, C \left(\frac{\tau_{k}-\tau^3_{k}}{\tau_{k}^{3}-\tau_{k-1}}\right)^{\alpha-\frac{1}{2}}  \left( \int_0^{s_{k,\ell_k^*+1}-\tau_{k-1}}(1+t)^{2\alpha}|g_\omega(\tau_{k-1}+t)|^2dt\right)^{\frac{1}{2}} \\
\le  & \, C\left(\frac{c_2}{c_3}\varsigma^{-k} 2^{-(\beta-1)\ell_k^*}\right)^{\frac{1}{2}-\alpha} \left( \int_0^{s_{k,\ell_k^*+1}-\tau_{k-1}}(1+t)^{2\alpha}|g_\omega(\tau_{k-1}+t)|^2dt\right)^{\frac{1}{2}} \\
\end{split}
\end{equation*}
where we used \eqref{durationstep3} and \eqref{eq:skuv} in the second inequality (By Remark \ref{rem:valeurc2}, $c_2=(C_K)^{\frac{1}{2\alpha}}$). 
But by Lemmas  \ref{lemma:step1}$(b)$ and \ref{lemme:step2.2}, for all $K>0$ and $\alpha\in(0,H)$, there exists $C>0$ such that 
$$  \left(\int_0^{s_{k,\ell_k^*+1}-\tau_{k-1}}(1+t)^{2\alpha}|g_\omega(\tau_{k-1}+t)|^2dt\right)^{\frac{1}{2}}\le C(1+2^{\alpha (\ell_k^*+1)}(\ell_k^*+1)).$$
By the condition $\beta>(1-2\alpha)^{-1}$ which ensures that $(\beta-1)(\frac{1}{2}-\alpha)>\alpha$, it follows that for every $K>0$ and $\alpha\in(0,1/2)$, there exists another constant $C>0$ (depending on $H$, $K$, $\alpha$ and $\beta$) such that 
$$\forall k\ge0,\quad u_{k+1}-u_k\le C(\frac{c_2}{c_3}\varsigma^{-k})^{\frac{1}{2}-\alpha}.$$
Choosing $c_3$ large enough in order that $C\sum_{k\ge1}(\frac{c_2}{c_3}\varsigma^{-k})^{\frac{1}{2}-\alpha}\le 1$ yields $\sup_{k\ge0} u_k\le 1$.
\end{proof}

\section{Proof of Theorem \ref{theo:principal}}\label{sec:prooftheoprinc}
Let $\alpha\in(0,1/2)$. We enforce  the assumptions of Proposition \ref{prop:imp1} and \ref{prop:minokalp}.
Assume that $X^1$ and $X^2$ have initial distributions $\mu_0$ and $\mu$ respectively, where $\mu$ denotes an  invariant distribution. {First, denoting by $\bar\mu$ its first marginal, we recall
that $\int |x|^r\bar{\mu}(dx)<+\infty$ for any positive $r$ (see Proposition 4.6 of \cite{hairer2} if $b$ is Lipschitz continuous or Proposition 3 of \cite{cohen-panloup-tindel} otherwise). It is therefore enough to show that 
 for  any initial condition $\tilde{\mu}$  of $(X^1,X^2)$  satisfying 
$\tilde{\mu}(|x_1|^r+|x_2|^r)<+\infty$  for some  $r>0$,  for each $\varepsilon>0$ there is $  C_\varepsilon>0$  such that $  \PE(\tau_\infty>t) \leq C_\varepsilon t^{-(\frac{1}{8}-\varepsilon)} $. }\smallskip

\noindent Set $k^*:=\inf\{k\ge1, \Delta \tau_k=+\infty\}.$ We have
\bqn\label{eq:doizpzoipz} \PE(\tau_\infty>t)=\PE(\tau_0+\sum_{k=1}^{+\infty}\Delta \tau_k 1_{k^*>k}>t)\le \PE(\tau_0>\frac{t}{2})+\PE(\sum_{k=1}^{+\infty}\Delta \tau_k 1_{k^*>k}>\frac{t}{2}).
\eqn 
Taking $\Psi$ such that $\Psi(x)\le C(1+|x|^r)$ (which is possible by Remark \ref{reducPsi}), by an argument  similar   to that of \eqref{eq:borntau0}  we deduce the  existence  of $C>0$ and $\gamma_0>0$ such that
$$ \PE(\tau_0>\frac{t}{2})\le C \exp(-\gamma_0 t).$$ 
Now, let us focus on the second term on the right-hand side of \eqref{eq:doizpzoipz} and let $p\in(0,\alpha/\beta)\subset (0,1)$. By the Markov inequality and the elementary inequality $|u+v|^p\le|u|^p+|v|^p$, 
\begin{align*}
\PE(\sum_{k=1}^{+\infty}\Delta \tau_k 1_{k^*>k}>\frac{t}{2})&\le\frac{C}{t^p}\sum_{k=1}^{+\infty}\ES[ |\Delta \tau_k|^p|1_{\{k^*>k\}}]\\
&\le\frac{C}{t^p}\sum_{k=1}^{+\infty}\ES[ \ES[|\Delta \tau_k|^p|1_{\{\Delta \tau_k<+\infty\}}|{\cal F}_{\tau_{k-1}}]1_{\tau_{k-1}<+\infty}].
\end{align*}
Let us bound deterministically the above conditional expectations. On the one hand,
if Step 1 fails (including the case where $\omega\in \Omega_{K,\alpha,\tau_{k-1}}^c$), $\Delta \tau_k=1+c_3\varsigma^k$ where $\varsigma>1$ can be chosen arbitrarily. On the other hand, by Lemma \ref{lemme:step2.2}, we have for every $\ell\ge2$,
\begin{align*}
\PE({\cal A}_{k,\ell}|{\cal F}_{\tau_{k-1}}\cap\{\tau_{k-1}<+\infty\})\le  2^{-\alpha\ell}.
\end{align*}
Since by construction, $\Delta \tau_k\le  C\varsigma^k  2^{\beta{\ell}}$ (with $\beta>(1-2\alpha)^{-1}$) on ${\cal A}_{k,\ell}$, this yields
$$\ES[|\Delta \tau_k|^p|1_{\{\Delta \tau_k<+\infty\}}|{\cal F}_{\tau_{k-1}}\cap\{\tau_{k-1}<+\infty\}]\le C\varsigma^{k p} \left(\sum_{\ell=1}^{+\infty} 2^{(\beta p-\alpha)\ell}\right)\le C \varsigma^{kp}.$$
Thus, for every $p\in(0,\alpha(1-2\alpha))$,
$$\PE(\sum_{k=1}^{+\infty}\Delta \tau_k 1_{k^*>k}>\frac{t}{2})\le \frac{C}{t^p}\sum_{k=1}^{+\infty} \varsigma^{kp} \PE(k^*>k-1).$$
But 
{
$$\PE(k^*>k-1)=\prod_{m=1}^{k-1} \PE({\cal E}_{m}|{\cal E}_{m-1})=\prod_{m=1}^{k-1} (1-\PE({\cal E}_{m}^c|{\cal E}_{m-1}))$$
and by Propositions \ref{prop:imp1} and \ref{prop:minokalp} the latter applied with  (for instance)  $\varepsilon=1/2$, we have for every $m\ge1$,
$$ \PE({\cal E}_{m}^c|{\cal E}_{m-1})\ge \PE(\Delta \tau_m=+\infty|\Omega_{K,\alpha,\tau_{m-1}})\PE(\Omega_{K,\alpha,\tau_{m-1}}|{\cal E}_{m-1})\ge \frac{\delta_0}{2}$$
where $\delta_0$ is a positive number depending on $K_{\frac{1}{2}}$. It follows that
$$\PE(k^*>k-1)\le (1-\frac{\delta_0}{2})^{k-1}.$$
As a consequence,
$$\sum_{k=1}^{+\infty}\varsigma^{kp} \PE(k^*>k-1)\le \sum_{k=1}^{+\infty} \varsigma^{pk}(1-\frac{\delta_0}{2})^{k-1}<+\infty$$
if $\varsigma$ is chosen in such a way that $\varsigma^p<   (1-\frac{\delta_0}{2})^{-1}$ (This is possible since $\varsigma$ is an arbitrary number greater than $1$). Finally, for every $\alpha\in(0,1/2)$, for every $p\in\alpha(1-2\alpha)$, there exists $C>0$ such that
$\PE(\tau_\infty>t)\le C t^{-p}$. To conclude the proof, it  only remains to optimize in $\alpha$.
}

\bigskip

{ {\bf Acknowledgements: } J.Fontbona acknowledges   support of Fondecyt Grant 1150570,  Basal-CONICYT Center for Mathematical Modeling (CMM), and  Millenium Nucleus  NC120062. F.Panloup  acknowledges partial  support of CNRS, INSA Toulouse  and of CMM  during a six months visit at Santiago in 2013-2014 where part of this work was carried out. }

\appendix
\section{Control of $(X_t^1)_{t\in[\tau_k,\tau_k+1]}$ under $(K,\alpha)$-admissibility} 
We show the first point of the proof of Lemma \ref{lemma:step1}. Let $K$ and $\tilde{K}$ denote some positive constants. Let $\omega\in\Omega_{K,\alpha,\tau_k}$ and assume that $\|W^1\|_{\frac{1}{2}-{\varepsilon_\theta}}^{\tau_k,\tau_k+1}\le \tilde{K}$ with ${\varepsilon_\theta}=\frac{H-\theta}{2}$.
First, we bound $\|B^1\|_\theta^{\tau_k,\tau_k+1}$. With the notations introduced at the beginning of Section \ref{subsec:hpdeux},
\begin{align*}
\|B^1\|_\theta^{\tau_k,\tau_k+1}&\le \sup_{\tau_{k}\le s<t\le \tau_{k}+1}\frac{1}{t-s}\left|\int_{-\infty}^{\tau_k-1} (t-r)^{H-\frac{1}{2}}-(s-r)^{H-\frac{1}{2}} dW_r^1\right|\\
&+\sum_{m=1}^3\|\Gamma_m\|_\theta^{\tau_k,\tau_k+1}.
\end{align*}
The first right-hand term is bounded by $K$ since $\omega\in \Omega_{K,\alpha,\tau_k}$.
As concerns that of the second line, we deduce from the end of the proof of Proposition  \ref{prop:contdubruit} (see $e.g.$ \eqref{eq:contg1}) that 
$$\forall m\in\{1,2,3\},\quad \|\Gamma_m\|_\theta^{\tau_k,\tau_k+1}\le C(1+\|W^1\|_{\frac{1}{2}-{\varepsilon}_\theta}^{\tau_k-1,\tau_k+1}).\quad $$
Under the  assumptions, $\|W^1\|_{\frac{1}{2}-{\varepsilon}_\theta}^{\tau_{k-1},\tau_k}\le K$ ($(K,\alpha)$-admissibility) and 
$\|W^1\|_{\frac{1}{2}-{{\varepsilon}_\theta}}^{\tau_k,\tau_k+1}\le \tilde{K}$. It follows that
\bqn\label{dopqipoq}
\forall m\in\{1,2,3\},\quad\|B^1\|_\theta^{\tau_k,\tau_k+1}\le C_{K,\tilde{K}}
\eqn
where $C_{K,\tilde{K}}$ is a finite deterministic constant which does not depend on $k$.

In order to conclude the proof, it is now enough to bound $\sup_{t\in[0,1]}|X_{\tau_{k}+t}|$ with respect to $|X_{\tau_k}^1|$ and $\|B^1\|_\theta^{\tau_k,\tau_k+1}$. This point is a classical property of fractional driven SDE but we prove it for the sake of completeness. First, note that Steps $1$, $2$ and $3$ of the proof of Proposition \ref{prop:lyapounov}
still hold under the assumptions of Lemma . Set $F(x)={1+|x|^2}$. Let $\tau_{k}\le s<t\le \tau_{k}+1$ such that 
$c_0(1+\|B^1\|_\theta^{\tau_{k},\tau_{k}+1})\le 1/2$. By the change of variable formula,
\begin{align*}
F(X_{t})&=F(X_s)+\int_{s}^t (\nabla F|b)(X_u)du+\int_s^t (\nabla F(X_u)|\sigma (X_u) dB_u^1)\\
&\le F(X_s)(1+C(t-s))+|\int_s^t (\nabla F(X_u)|\sigma (X_u) dB_u^1|,
\end{align*}
where in the second line, we used that $(\nabla F|b)(x)\le C F(x)$ (since $b$ is a sublinear function) and Step 3 of the proof of Proposition \ref{prop:lyapounov}. 
The functions $\nabla F$ and $\sigma$ being Lipschitz continuous and $\sigma$ being also bounded, we obtain similarly to Step 4 of the proof of Proposition \ref{prop:lyapounov} (see \eqref{eq:contint}) that  if $t-s\le\eta$ defined by \eqref{eq:eta} (replacing $0$ and $1$ by $\tau_k$ and $\tau_{k}+1$ respectively),
\bqne
|\int_s^t (\nabla F (X_u)|\sigma(X_u) dB_u^1)|\le C (t-s)^{\frac{1}{2}+\theta} F(X_s)+ \tilde{\beta}.
\eqne
Then, it follows that for every $\tau_{k}\le s<t\le \tau_{k}+1$ such that $t-s\le\eta$
$$F(X_t)\le F(X_s)(1+C(t-s))+\tilde{\beta}.$$
An iteration of this inequality on the sequence $(\tau_k+ \ell\eta)_{\ell\ge 1 }$ yields
\begin{align*}
\forall \ell\in,\ldots,\lfloor \frac{1}{\eta}\rfloor \},\quad &F(X_{\tau_k+\ell\eta})\le F(X_{\tau_k})(1+C\eta)^{\ell}+ \tilde{\beta}\sum_{u=0}^{\ell-1} (1+C\eta)^{u}\\
&\le \exp(C) (F(X_{\tau_k})+\frac{\tilde\beta}{C\eta})\le \exp(C)(F(X_{\tau_k})+\tilde{C}{\tilde\beta}(1+\|B^1\|_{\theta}^{\tau_k,\tau_k+1})^{\frac{4}{2\theta-1}})
\end{align*}
where in the last inequality, we used that $\eta\ge (1+\|B^1\|_\theta)^{\tau_k,\tau_k+1})^{-\frac{1}{\tilde{\theta}}}$. Applying again Step 3 of the proof of Proposition \ref{prop:lyapounov} yields the existence of another constant $C$ (which does not depend on $k$) such that 
\bqn\label{eq:compx}
\sup_{t\in[\tau_k,\tau_k+1]} F(X_t)\le C (F(X_{\tau_k})+{\tilde\beta}(1+\|B^1\|_{\theta}^{\tau_k,\tau_k+1})^{\frac{4}{2\theta-1}}).
\eqn
The result follows since, by \eqref{dopqipoq}, the right-hand side is bounded by a deterministic constant depending only on $K$, $\tilde{K}$ and $\theta$ on the set $\Omega_{K,\alpha,\tau_k}\cap\{\|W^1\|_{\frac{1}{2}-{\varepsilon_\theta}}^{\tau_k,\tau_k+1}\le \tilde{K}\}$.
\section{Proof of \eqref{eq:contmemoire}}
It is enough to prove that there exists $C>0$ such that for every $k$ and $K$, $\ES[\varphi_{\tau_k,{{\varepsilon}_\theta}}(W^j(\omega))|{\cal E}_k]\le C$.
The fact that 
$$\ES[\sup_{(s,t) \tau_k\le s<t\le \tau_k+1} \frac{1}{t-s}\left|\int_{-\infty}^{\tau_k-1} (t-r)^{H-\frac{1}{2}}-(s-r)^{H-\frac{1}{2}} dW_r^j\right| |{\cal E}_k]$$
is bounded by a constant which does not depend on $k$ follows from Lemma \ref{lemme:lambdamk}. More precisely, if $k\ge1$ this property is a consequence of \eqref{eq:pastbf2} combined with the fact that  $\{(s,t), \tau_k\le s<t\le \tau_k+1\}\subset \{(s,t), \tau_{k-1}\le s<t\le s+1\}.$
If $k=0$, it corresponds to Case 3 of  the proof of Lemma \ref{lemme:lambdamk}. For the second part, if $k\ge1$, we deduce from Cauchy-Schwarz Inequality that
$$\ES[\|W^j\|_{\frac{1}{2}-\varepsilon_\theta}^{\tau_{k}-1,\tau_k}|{\cal E}_k]\le \ES[(\|W^j\|_{\frac{1}{2}-\varepsilon_\theta}^{\tau_{k}-1,\tau_k})^2|{\cal E}_{k-1}]^{\frac{1}{2}}\PE({\cal E}_k|{\cal E}_{k-1})^{-\frac{1}{2}}.$$
But using that $(W_j)_{t\in[\tau_k-1,\tau_k]}$ is independent of ${\cal E}_{k-1}$ and that $\PE({\cal E}_k|{\cal E}_{k-1})\ge\delta_1$, it follows that
$$\ES[\|W^j\|_{\frac{1}{2}-\varepsilon_\theta}^{\tau_{k}-1,\tau_k}|{\cal E}_k]\le C_\theta \delta_1^{-\frac{1}{2}}$$
where $C_\theta:= \ES[\left(\|W^j\|_{\frac{1}{2}-\varepsilon_\theta}^{0,1}\right)^2]<+\infty$. This concludes the proof.
\bibliographystyle{plain}
\bibliography{biblio_fontbona_panloup}

\begin{thebibliography}{10}

\bibitem{Arnold98}
Ludwig Arnold.
\newblock {\em Random dynamical systems}.
\newblock Springer Monographs in Mathematics. Springer-Verlag, Berlin, 1998.

\bibitem{cohen-panloup}
Serge Cohen and Fabien Panloup.
\newblock Approximation of stationary solutions of {G}aussian driven stochastic
  differential equations.
\newblock {\em Stochastic Process. Appl.}, 121(12):2776--2801, 2011.

\bibitem{cohen-panloup-tindel}
Serge Cohen, Fabien Panloup, and Samy Tindel.
\newblock Approximation of stationary solutions to {SDE}s driven by
  multiplicative fractional noise.
\newblock {\em Stochastic Process. Appl.}, 124(3):1197--1225, 2014.

\bibitem{Coutin12}
Laure Coutin.
\newblock Rough paths via sewing lemma.
\newblock To appear in ESAIM PS, 2012.

\bibitem{crauel}
Hans Crauel.
\newblock Non-{M}arkovian invariant measures are hyperbolic.
\newblock {\em Stochastic Process. Appl.}, 45(1):13--28, 1993.

\bibitem{DownMeynTweedie}
D.~Down, S.P. Meyn, and R.L. Tweedie.
\newblock Exponential and uniform ergodicity of markov processes.
\newblock {\em The Annals of Probability}, 23:1671--1691, 1995.

\bibitem{GKN09}
Mar{\'{\i}}a~J. Garrido-Atienza, Peter~E. Kloeden, and Andreas Neuenkirch.
\newblock Discretization of stationary solutions of stochastic systems driven
  by fractional {B}rownian motion.
\newblock {\em Appl. Math. Optim.}, 60(2):151--172, 2009.

\bibitem{Gua06}
Paolo Guasoni.
\newblock No arbitrage under transaction costs, with fractional {B}rownian
  motion and beyond.
\newblock {\em Math. Finance}, 16(3):569--582, 2006.

\bibitem{hairer}
Martin Hairer.
\newblock Ergodicity of stochastic differential equations driven by fractional
  {B}rownian motion.
\newblock {\em Ann. Probab.}, 33(2):703--758, 2005.

\bibitem{hairer2}
Martin Hairer and Alberto Ohashi.
\newblock Ergodic theory for {SDE}s with extrinsic memory.
\newblock {\em Ann. Probab.}, 35(5):1950--1977, 2007.

\bibitem{hairer-pillai}
Martin Hairer and Natesh~S. Pillai.
\newblock Regularity of laws and ergodicity of hypoelliptic {SDE}s driven by
  rough paths.
\newblock {\em Ann. Probab.}, 41(4):2544--2598, 2013.

\bibitem{Jeon-al11}
Jae-Hyung Jeon, Vincent Tejedor, Stas Burov, Eli Barkai, Christine
  Selhuber-Unkel, Kirstine Berg-S\o{}rensen, Lene Oddershede, and Ralf Metzler.
\newblock \textit{In Vivo} anomalous diffusion and weak ergodicity breaking of
  lipid granules.
\newblock {\em Phys. Rev. Lett.}, 106:048103, Jan 2011.

\bibitem{Kou08}
S.~C. Kou.
\newblock Stochastic modeling in nanoscale biophysics: subdiffusion within
  proteins.
\newblock {\em Ann. Appl. Stat.}, 2(2):501--535, 2008.

\bibitem{MVn68}
Benoit~B. Mandelbrot and John~W. Van~Ness.
\newblock Fractional {B}rownian motions, fractional noises and applications.
\newblock {\em SIAM Rev.}, 10:422--437, 1968.

\bibitem{mattingly02}
Jonathan~C. Mattingly.
\newblock Exponential convergence for the stochastically forced
  {N}avier-{S}tokes equations and other partially dissipative dynamics.
\newblock {\em Comm. Math. Phys.}, 230(3):421--462, 2002.

\bibitem{Nualart02}
David Nualart and Aurel R{\u{a}}{\c{s}}canu.
\newblock Differential equations driven by fractional {B}rownian motion.
\newblock {\em Collect. Math.}, 53(1):55--81, 2002.

\bibitem{Odde-al96}
David~J. Odde, Elly~M. Tanaka, Stacy~S. Hawkins, and Helen~M. Buettner.
\newblock Stochastic dynamics of the nerve growth cone and its microtubules
  during neurite outgrowth.
\newblock {\em Biotechnology and Bioengineering}, 50(4):452--461, 1996.

\bibitem{ruadulescu}
Sorin R{\u{a}}dulescu and Marius R{\u{a}}dulescu.
\newblock An application of {H}adamard-{L}\'evy's theorem to a scalar initial
  value problem.
\newblock {\em Proc. Amer. Math. Soc.}, 106(1):139--143, 1989.

\bibitem{young}
L.~C. Young.
\newblock An inequality of the {H}\"older type, connected with {S}tieltjes
  integration.
\newblock {\em Acta Math.}, 67(1):251--282, 1936.

\bibitem{zahle}
M.~Z{\"a}hle.
\newblock Integration with respect to fractal functions and stochastic
  calculus. {I}.
\newblock {\em Probab. Theory Related Fields}, 111(3):333--374, 1998.

\end{thebibliography}
\end{document}